\documentclass[12pt,a4paper]{article}

\usepackage[utf8]{inputenc}
\usepackage[english]{babel}
\usepackage{amssymb,mathtools,amsmath,amsthm,stmaryrd}
\usepackage[all,cmtip,2cell]{xy} 
\usepackage[draft=false, hidelinks, linkcolor = blue]{hyperref}
\usepackage{cleveref}
\usepackage{tensor}
\usepackage{url}
\usepackage{bookmark}
\usepackage{graphicx}
\usepackage{lmodern}
\usepackage[left=2.5cm,right=2.5cm,top=2.5cm,bottom=2.5cm]{geometry}

\def\defthm#1#2#3#4{
  \newtheorem{#1}[theorem]{#3}
  \newtheorem*{#1*}{#3}
  \newtheorem{#2}[theorem]{#4}
  \newtheorem*{#2*}{#4}
  \crefname{#1}{#3}{#4}
  \crefname{#2}{#4}{#4}  
}

\newtheoremstyle{mythm}%
{10pt}
{}
{\itshape}
{}
{\bf}
{.}
{.5em}
{}%

\newtheoremstyle{mydef}%
{10pt}
{3pt}
{}
{}
{\bf}
{.}
{.5em}
{}%

\newtheoremstyle{myrmk}%
{10pt}
{3pt}
{}
{}
{\bf}
{.}
{.5em}
{}%

\theoremstyle{mythm}
\newtheorem{theorem}{Theorem}[section]
\newtheorem*{theorem*}{Theorem}

\defthm{corollary}{corollaries}{Corollary}{Corollaries}
\defthm{lemma}{lemmata}{Lemma}{Lemmata}
\defthm{proposition}{propositions}{Proposition}{Propositions}
\defthm{axiom}{axioms}{Axiom}{Axioms}
\defthm{propdef}{props/defs}{Proposition/Definition}{Prositions/Definitions}
\defthm{defcor}{defscors}{Corollary/Definition}{Corollaries/Definitions}
\defthm{conjecture}{conjectures}{Conjecture}{Conjectures}

\theoremstyle{mydef}
\defthm{definition}{definitions}{Definition}{Definitions}

\theoremstyle{myrmk}
\defthm{relatedwork}{relatedworks}{Related work}{Related Works}
\defthm{acknowledgment}{acknowledgments}{Acknowledgment}{Acknowledgments}
\defthm{remark}{remarks}{Remark}{Remarks}
\defthm{motivation}{motivations}{Motivation}{Motivations}
\defthm{example}{examples}{Example}{Examples}
\defthm{question}{questions}{Question}{Questions}
\defthm{observation}{observations}{Observation}{Observations}
\defthm{claim}{claims}{Claim}{Claims}
\defthm{notation}{notations}{Notation}{Notations}

\newtheorem*{replemmax}{\reptitle}
 {\end{replemmax}}

\newtheorem*{repthmx}{\reptitle}
\newenvironment{repthm}[1]{%
 \def\reptitle{Theorem \ref*{#1}}%
 \begin{repthmx}}%
 {\end{repthmx}}

\newtheorem*{repcorx}{\reptitle}
 {\end{repcorx}}

\crefname{section}{Section}{Sections}
\crefname{theorem}{Theorem}{Theorems}

\makeatletter
\renewenvironment{proof}[1][\proofname] {\par\pushQED{\qed}\normalfont\topsep6\p@\@plus6\p@\relax\trivlist\item[\hskip\labelsep\bf#1\@addpunct{.}]\ignorespaces}{\popQED\endtrivlist\@endpefalse}
\makeatother

\newcommand{\blank}{\mbox{\hspace{3pt}\underline{\ \ }\hspace{2pt}}}
\newcommand{\sprime}{^{\prime}}

\newcommand{\pbs}{\scalebox{1.5}{\rlap{$\cdot$}$\lrcorner$}}
\newcommand{\pos}{\rotatebox[origin=c]{180}{\pbs}}

\newcommand{\bBox}{\mathbin{\Box}}

\title{Univalence and completeness of Segal objects}
\author{Raffael Stenzel\thanks{The author acknowledges the support of the Grant Agency of the Czech Republic under the
grant 19-00902S.}}
\renewcommand\footnotemark{}

\newcommand{\Address}{{
  \bigskip
  \footnotesize

\textsc{Department of Mathematics and Statistics, Masaryk University, Kotl\'{a}\v{r}sk\'{a} 2, Brno 61137, Czech Republic}\par\nopagebreak
  \textit{E-mail address}: \texttt{stenzelr@math.muni.cz}
}}

\begin{document}
\maketitle
\abstract{
Univalence, originally a type theoretical notion at the heart of Voevodsky's Univalent Foundations Program, 
has found general importance as a higher categorical property that characterizes descent and hence classifying maps in
$(\infty,1)$-categories. Completeness is a property of Segal spaces introduced by Rezk that characterizes those Segal 
spaces which are $(\infty,1)$-categories. In this paper, first, we make rigorous an analogy between univalence and 
completeness that has found various informal expressions in the higher categorical research community to date, and 
second, study its ramifications. 

The core aspect of this analogy can be understood as a translation between internal and external notions, 
motivated by model categorical considerations of Joyal and Tierney. As a result, we characterize the internal notion of 
univalence in logical model categories by the external notion of completeness defined as the right Quillen condition of suitably 
indexed Set-weighted limit functors.

Furthermore, we extend the analogy and show that univalent completion in the sense of van den Berg and Moerdijk 
translates to Rezk-completion of associated Segal objects as well. Motivated by these correspondences, we 
exhibit univalence as a homotopical locality condition whenever univalent completion exists.
}

\section{Introduction}
Complete Segal spaces as models for $(\infty,1)$-category theory were introduced by Charles Rezk in \cite{rezk}. 
Segal spaces in general can be thought of as homotopy-coherent compositional structures enriched in the homotopy theory 
of spaces. By definition, they are simplicial spaces $X\colon\Delta^{op}\rightarrow\mathbf{S}$ equipped with a 
contractible space of composition operations
\begin{align}\label{equintro1}
\circ\colon X_1\times_{X_0}X_1\rightarrow X_1
\end{align}
where $X_0$ is to be thought of as the space of objects, $X_1$ as the space of arrows and the higher $X_n$ as the spaces 
of higher compositions. A Segal space $X$ is said to be complete if the homotopical structure on $X_0$ induced by the 
internally invertible arrows in $X$, and the canonical homotopical structure on $X_0$ induced by the ambient homotopy 
theory of $\mathbf{S}$, coincide. In more detail, this means that a suitably defined object $X^{\simeq}$ of internal 
equivalences in $X$ is a path-object for $X_0$.

Rezk has constructed two model structures S and CS on the category $\mathbf{S}^{\Delta^{op}}$ of simplicial spaces whose 
fibrant objects are exactly the Segal spaces and complete Segal spaces, respectively. Joyal and Tierney have 
shown in \cite{jtqcatvsss} that the model category $(\mathbf{S}^{\Delta^{op}},\mathrm{CS})$ is Quillen equivalent to the 
Joyal model structure $(\mathbf{S},\mathrm{QCat})$ for quasi-categories in a canonical way. Joyal and Tierney's 
construction of the Quillen equivalence to $(\mathbf{S},\mathrm{QCat})$ has been generalized to ``theories 
of $(\infty,1)$-categories'' by To\"{e}n in \cite{toeninftycats}. All these model categorical results about complete 
Segal spaces rely on a specific technical definition of completeness via choice of a suitably defined object of 
internal equivalences as mentioned in the paragraph above. This is the definition we will use to define complete Segal 
objects in general model categories.

Therefore, we note that the Segal conditions that induce the horizontal composition operation in (\ref{equintro1}) can 
be expressed for simplicial objects in any model category $\mathbb{M}$ in a straight-forward fashion. We thereby 
obtain a notion of Segal objects in all model categories. All such Segal objects $X$ (in fact all simplicial objects in
$\mathbb{M}$) can be \emph{externalized} by computing their weighted limit functor together with (the opposite of) their 
associated $\Delta^{op}$-indexed cotensor
\begin{align}\label{equintrocplt}
X/\blank\colon\mathbb{M}^{op}\rightarrow\mathbf{S}.
\end{align}
Joyal and Tierney's result in \cite[4]{jtqcatvsss} states that the horizontal projection
\[\iota_1^{\ast}\colon\mathbf{S}^{\Delta^{op}}\rightarrow\mathbf{S}\]
yields a Quillen pair -- and, in fact, a Quillen 
equivalence -- between $(\mathbf{S}^{\Delta^{op}},\mathrm{CS})$ and $(\mathbf{S},\mathrm{QCat})$, but that it does not 
do so when defined on $(\mathbf{S}^{\Delta^{op}},\mathrm{S})$ only. This can be rephrased by stating that a Segal space 
$X$ is complete if and only if its associated cotensor (\ref{equintrocplt}) yields a right Quillen functor of the form
\[X/\blank\colon(\mathbf{S},\mathrm{Kan})^{op}\rightarrow(\mathbf{S},\mathrm{QCat}),\]
where $\mathrm{Kan}$ denotes Quillen's standard model structure for Kan complexes.
We will define complete Segal objects in general model categories $\mathbb{M}$ to be those Segal 
objects $X$ such that their associated cotensor (\ref{equintrocplt}) yields a right Quillen functor into the model 
category $(\mathbf{S},\mathrm{QCat})$. This \emph{external} notion defined by a 1-categorical weighted limit appears to 
be model categorical in nature and as such not necessarily to be tangible on underlying $(\infty,1)$-categories. Yet it is the notion 
that underlies not only the fundamental results of Joyal and Tierney, as well of To\"{e}n more generally, but also Rezk's 
construction of the model structure for complete Segal spaces in the first place (\cite[Theorem 6.2, 12]{rezk}). The same holds for 
Barwick's construction of the model structures for $n$-fold complete Segal spaces (\cite[2.3]{barwickthesis}).
We will thus further motivate and study this notion of completeness in Section~\ref{seccompleteness} in a general model 
category. In fact, we will see that each of the three partial conditions on a simplicial object $X$ in $\mathbb{M}$ -- Reedy fibrancy, 
the Segal conditions and completeness -- corresponds to a preservation property of the cotensor (\ref{equintrocplt}) with respect to a 
well-known class of maps in $\mathbf{S}$, see Proposition~\ref{lemmarQuillenchar}.

An \emph{internal} notion of completeness for Segal objects in model categories, or in left exact
$(\infty,1)$-categories in general, however is obtained by a natural generalization of univalence, by which we mean the 
following.

Univalence is a predicate in the internal language of type theoretic fibration 
categories $\mathbb{C}$. Every fibration $p\colon E\twoheadrightarrow B$ in $\mathbb{C}$ comes with an 
associated family of types $\mathrm{Equiv}(p)\twoheadrightarrow B\times B$ that consists of fiberwise homotopy 
equivalences $E(b)\rightarrow E(b\sprime)$ in $\mathbb{C}$ parametrized by pairs of objects $b,b\sprime:B$. A fibration
$p\colon E\twoheadrightarrow B$ is \emph{univalent} if its associated type $\mathrm{Equiv}(p)$ is a path 
object for the base $B$. Voevodsky's Univalence Axiom (\cite[Section 2.10]{hott}) states that the Tarski type universes 
$\tilde{U}\twoheadrightarrow U$ are univalent (whenever they exist). In syntactic terms, the axiom states that the 
canonical path-transport map
\[\big(A=_{U}B\big)\rightarrow \big(A\simeq B\big)\]
is an equivalence for all pairs of types $A,B$ (up to a slight abuse of the syntax). In this diction, generally, 
a type family $p\colon E\twoheadrightarrow B$ is univalent if and only if the canonical transport maps of the form
\begin{align}\label{equintro2}
tr_p(b,b\sprime)\colon\big(b=_B b\sprime\big)\rightarrow \big(E(b)\simeq E(b\sprime)\big)
\end{align}
are equivalences.
In Section~\ref{secunivsegal} we generalize this notion of univalence for fibrations in such categories $\mathbb{C}$ to 
a notion of univalence for Segal objects\footnote{For the sake of this generalization, the notion of Segal objects in 
this paper will assume a fibrancy condition that is weaker than Reedy fibrancy. It is hence slightly more general than 
the definition of Segal objects as known from the standard references.}
$X$ in $\mathbb{C}$ via a correspondingly defined object of internal 
equivalences. This internally defined notion of univalence is what in recent years often is meant by completeness in the 
literature. For the purposes of this paper, we will call this property univalence however and reserve the name 
completeness for the strict and external notion motivated above.

The aim of Section~\ref{secunivvscompl} therefore is to compare univalence and completeness accordingly defined and 
generalized in categories that combine the frameworks of Section~\ref{secunivsegal} and Section~\ref{seccompleteness}. 
These are the logical model categories (with an additional commonly assumed cofibrancy condition). The resulting 
juxtaposition of the internal and the external, or the weak and the strict, will boil down to the comparison of two 
types of equivalences associated to Segal objects in such model categories. We thus prove the following theorem and 
define all the notions involved. 

\begin{repthm}{intc=simpc}
Let $X$ be a Segal object in a logical model category $\mathbb{M}$ in which all fibrant objects are cofibrant. 
Then $X$ is univalent if and only if its Reedy fibrant replacement $\mathbb{R}X$ is complete, by which we mean that its 
$\Delta^{op}$-indexed cotensor yields a right Quillen functor of the form
\[\mathbb{R}X/\blank\colon\mathbb{M}^{op}\rightarrow(\mathbf{S},\mathrm{QCat}).\]
In particular, if $X$ is Reedy fibrant, then $X$ is univalent if and only if it is 
complete.
\end{repthm}
It follows that a fibration $p$ in the category $\mathbb{C}:=\mathbb{M}^f$ of fibrant objects is univalent 
in the syntactic sense (\ref{equintro2}) if and only if the Reedy fibrant replacement $X_p$ of its internal nerve $N(p)$ 
is complete in the above sense.

In Section \ref{secunivcompl} we extend the assignment $p\mapsto X_p$ to a functor of 
relative categories, defining a notion of BM-equivalences between fibrations inspired by \cite{univalentcompletion} 
and a notion of DK-equivalences between Segal objects inspired by \cite{rezk}, respectively.
This functor $X$ is then shown to preserve and reflect the respective completion procedures; that is, univalent 
completion of a fibration $p$ on the one hand -- originally introduced in \cite{univalentcompletion} in the case of Kan 
fibrations of simplicial sets -- and Rezk-completion of Segal objects on the other. Rezk has shown that complete Segal 
spaces are exactly those Segal spaces local with respect to all DK-equivalences, a result that was generalized by 
Barwick to a large class of model categories. Analogously, subject to some technical conditions, we show in
Section~\ref{secsubBMloc} that univalent fibrations are exactly those fibrations local with respect to all
BM-equivalences whenever univalent completion exists.

\begin{relatedwork*}
As the two notions of univalence and Rezk-completeness are often identified in the literature, it may be 
worth to point out the following. If one chooses to define completeness in terms of univalence,
then the reader may interpret the results of the first part of this paper 
as a characterization of completeness -- now an internal notion defined via some homotopy limit --
\begin{enumerate}
\item via a strict 1-categorical limit that is used in model categorical considerations such as those of Rezk and 
Barwick.
\item as the right Quillen condition of associated cotensors, motivated by the results of Joyal and Tierney.
\end{enumerate}

A connection of Rezk-completeness and univalence via a nerve construction $p\mapsto N(p)$ has been studied in 
\cite[Section 6]{rasekh} as well. In that paper, Rasekh develops a theory of complete Segal objects in
$(\infty,1)$-categories and \emph{defines} univalence of a map $p$ in a locally cartesian closed $(\infty,1)$-category
$\mathcal{C}$ as completeness of its associated Segal object $\mathcal{N}(p)$ (\cite[Definition 6.24]{rasekh}).
This $(\infty,1)$-categorical notion of completeness for Segal objects corresponds to what we call univalence of Segal 
objects via Proposition~\ref{defunivalenceso}. The identification of univalence of a fibration $p$ and univalence of 
its associated Segal object $N(p)$ in a type theoretic fibration category is justified by
Lemma~\ref{lemmasimpunivalence}. The fact that the $(\infty,1)$-categorical notion of completeness -- i.e.\ univalence -- coincides 
with the model categorical notion of completeness whenever $\mathcal{C}$ underlies a logical model category follows from 
Theorem~\ref{intc=simpc}. For this particular correspondence in fact right-properness is enough, see Remark~\ref{remrpmodcats}.

In \cite{aksrezkcompl}, the authors introduce a notion of Rezk-completion of precategories to categories (say, univalent 
precategories) in the syntax of Homotopy Type Theory. The definition of univalence for precategories naturally 
generalizes the definition of univalence for type families, and the associated notion of Rezk-completion is in spirit 
closely related to what we consider in Section~\ref{secsubrezkcompl}. Yet note that the definition of a precategory in 
Homotopy Type Theory is not the same as the definition of a Segal object in a type theoretic fibration category (see 
Remark~\ref{remintrezkcompl} for some more detail).
\end{relatedwork*}

\begin{acknowledgments*}
The author would like to thank his advisor Nicola Gambino for pointing him to a formal comparison of univalence and 
Rezk-completeness as worked out in Section~\ref{secunivvscompl} which lies at the core of this paper. Its implementation 
followed a discussion of his with Richard Garner on this issue at the Logic Colloquium 2016 in Leeds. The author is 
also grateful to John Bourke, Denis-Charles Cisinski, Peter Lumsdaine, Tim Porter, Nima Rasekh, Christian Sattler and 
Karol Szumi\l{}o for helpful comments and discussions. He is particularly grateful to the anonymous referee whose comments improved 
the exposition of results considerably. Large portions of this paper were part of the author's PhD 
thesis, supported by a University of Leeds 110 Anniversary Research Scholarship. It was expanded and finished 
with support of the Grant Agency of the Czech Republic under the grant 22-02964S.
\end{acknowledgments*}

\section{Univalence of simplicial objects}\label{secunivsegal}

In this section we discuss a notion of univalence for simplicial objects in type theoretic fibration 
categories $\mathbb{C}$ that naturally generalizes the notion of univalence for fibrations in $\mathbb{C}$. Thus, 
first, we explain what we mean by univalence for simplicial objects in $\mathbb{C}$, second, associate to a fibration
$p$ in $\mathbb{C}$ its nerve $N(p)\in\mathbb{C}^{\Delta^{op}}$, and third, show that $p$ is a univalent fibration in
$\mathbb{C}$ if and only if $N(p)$ is a univalent simplicial object in $\mathbb{C}$. 

Throughout this section $\mathbb{C}$ is a fixed type theoretic fibration category in the sense of
\cite[Definition 2.1]{shulmaninv}) defined as follows. 
\begin{definition}\label{defttfibcat}
A \emph{type theoretic fibration category} is a pair $(\mathbb{C},\mathcal{F})$ where $\mathbb{C}$ is a 
category and $\mathcal{F}\subseteq\mathbb{C}$ is a wide and replete subcategory whose morphisms are called 
\emph{fibrations} -- depicted by arrows of the form $E\twoheadrightarrow B$ -- such that the following hold.
\begin{enumerate}
\item $\mathbb{C}$ has a terminal object $1$.
\item $\mathcal{F}\subseteq\mathbb{C}$ contains all morphisms with codomain $1$.
\item All pullbacks of fibrations exist and are fibrations.
\item Let $\mathcal{AC}$ be the class of morphisms with the left lifting property with respect to all 
fibrations. Elements of $\mathcal{AC}$ are called acyclic cofibrations and depicted by arrows of the form
$A\overset{\sim}{\hookrightarrow}B$. Then every morphism factors as an acyclic cofibration followed by a fibration.
\item Given an object $B$ in $\mathbb{C}$, let $\mathcal{F}/B$ denote the full subcategory of $\mathbb{C}/B$ 
whose objects are the fibrations with codomain $B$. Then for every fibration $p\colon E\twoheadrightarrow B$, the 
pullback functor $p^{\ast}\colon\mathcal{F}/B\rightarrow\mathcal{F}/E$ has a right adjoint $\prod_p$.
\end{enumerate}
\end{definition}
In other words, a type theoretic fibration category $\mathbb{C}$ is a tribe with $\pi$-clan structure in the sense of 
Joyal (\cite[2.4,3]{joyaltribes}, \cite[Definition 2.10, Lemma 2.15]{kapulkinszumilo}).

For the readers convenience, we proceed to summarise some basic facts about type theoretic fibration categories in 
general, before we continue with the definition of Segal objects therein.

\paragraph{A quick overview of the homotopy theory}
Condition 5 in Definition~\ref{defttfibcat} implies that acyclic cofibrations are stable under 
pullback along fibrations. The class of fibrations in a type theoretic fibration category $\mathbb{C}$ generates a weak 
factorization system $(\mathcal{AC},\mathcal{AC}^{\pitchfork})$ on $\mathbb{C}$ where $\mathcal{AC}^{\pitchfork}$ 
denotes the class of morphisms with the right lifting property with respect to all acyclic cofibrations. In 
particular we have $\mathcal{F}\subseteq\mathcal{AC}^{\pitchfork}$.
The existence of pullbacks along fibrations together with the existence of factorizations as required in
Definition~\ref{defttfibcat}.4 induces a notion of \emph{path-object} in type theoretic fibration categories
$\mathbb{C}$. Recall that a path-object for a fibration $E\twoheadrightarrow B$ is a factorization
$E\overset{\sim}{\hookrightarrow}P_{B}E\twoheadrightarrow E\times_{B}E$ of the diagonal
$E\rightarrow E\times_{B}E$ over $B$ into an acyclic cofibration followed by a fibration.
We will sometimes denote the left map by $r_E\colon E\overset{\sim}{\hookrightarrow}P_{B}E$.
We hence obtain a notion of \emph{right homotopy} and thereby an according notion of \emph{homotopy-equivalence} in
$\mathbb{C}$ as explained in \cite[Section 3]{shulmaninv}. We will sometimes write $f\sim g$ (or $f\sim_H g$) to denote 
that two maps $f,g\colon A\rightarrow B$ are homotopic (via a specified homotopy $H\colon A\rightarrow PB$). An
\emph{acyclic fibration} is a fibration that is also a homotopy-equivalence. A type theoretic fibration category
$\mathbb{C}$ together with its collection $\mathcal{F}$ of fibrations and its collection $\mathcal{W}$ of homotopy-equivalences forms 
a category of fibrant objects in the sense of Brown (\cite[Theorem 3.13]{shulmaninv}). In particular, acyclic fibrations are
pullback-stable, and hence so are homotopy-equivalences between fibrations in $\mathbb{C}$ by Ken 
Brown's Lemma. In particular, it follows that path-objects of fibrations in $\mathbb{C}$ are pullback-stable.

\paragraph{A quick overview of the associated syntax}
Every type theoretic fibration category $\mathbb{C}$ comes canonically equipped with the structure of a full comprehension category (\cite{jacobscompcats})
\[\xymatrix{
\mathcal{F}\ar@{^(->}[r]\ar@/_/@{->>}[dr] & \mathbb{C}^{[1]}\ar@{->>}[d] \\
 & \mathbb{C}
}\]
denoted in short by $\mathbb{C}$ as well, where $\mathcal{F}$ consists of the fibrations in $\mathbb{C}$. In fact, as 
shown in \cite[Section 4.2]{shulmaninv}, every type theoretic fibration category $\mathbb{C}$ has an internal type 
theory $\mathcal{T}_{\mathbb{C}}$ that supports dependent product types, dependent function types and identity types, 
together with an interpretation of $\mathcal{T}_{\mathbb{C}}$ in $\mathbb{C}$. More precisely, by the results of 
\cite{lwlocaluniv}, there is a \emph{split} full comprehension category
\[\xymatrix{
\mathcal{F}_!\ar@{^(->}[r]\ar@/_/@{->>}[dr] & \mathbb{C}^{[1]}\ar@{->>}[d] \\
 & \mathbb{C}
}\]
denoted by $\mathbb{C}_!$ such that
\begin{enumerate}
\item there is an equivalence $\mathbb{C}_!\rightarrow\mathbb{C}$ of full comprehension categories which is the 
identity on the underlying fibered category $\mathbb{C}^{[1]}\twoheadrightarrow\mathbb{C}$, and
\item there is a strict interpretation functor
$\llbracket\cdot\rrbracket\colon\mathcal{T}_{\mathbb{C}}\rightarrow\mathbb{C}_!$ that preserves dependent product types, 
dependent function types and identity types.
\end{enumerate}
Contexts $\vdash(b:B)\text{ }\mathtt{ctx}$ of length 1 in the type theory $\mathcal{T}_{\mathbb{C}}$ -- presented as a 
contextual category (\cite{cartmell}) -- are given by 
objects $B$ in $\mathbb{C}$ with $\llbracket b:B\rrbracket=B$. A type family $E$ in a context $\Gamma=(b_1:B_1,\dots,b_n:B_n)$, 
denoted in the syntax by
\[\Gamma\vdash E\text{ }\mathtt{type},\]
or at times by $x:\Gamma\vdash E(x)\text{ }\mathtt{type}$ to highlight free variables, corresponds to a span
\[\xymatrix{
 & E\ar@{->>}[d]^p \\
\llbracket\Gamma\rrbracket\ar[r]_{f} & B \\
}\]
in $\mathbb{C}$, where $p\colon E\twoheadrightarrow B$ is a fibration.\footnote{The name $f$ and the 
fibration $p$ are left implicit in the notation. More carefully, one would denote the corresponding type family by
$\Gamma\vdash (f,p) \text{ }\mathtt{type}$.} It hence yields a context extension of the form
\[\vdash (\Gamma,e:E)\text{ }\mathtt{ctx}\]
such that the display map $(\Gamma,e:E)\twoheadrightarrow\Gamma$ is the pullback $f^{\ast}p$ under
$\llbracket\cdot\rrbracket$ in $\mathbb{C}_!$. If we denote $\Gamma_m=(b_1:B_1,\dots,b_m:B_m)$ for $m\leq n$, we thus obtain a 
sequence of fibrations $\llbracket\Gamma\rrbracket\twoheadrightarrow\llbracket\Gamma_{n-1}\rrbracket\twoheadrightarrow\dots\twoheadrightarrow B_1$ in $\mathbb{C}$. A 
context of the form $(b_1:B,b_2:B)$ will be abbreviated by $b_1,b_2:B$. A general morphism $\Gamma\rightarrow\Delta$ of contexts in $\mathcal{T}_{\mathbb{C}}$ is just a morphism
$\llbracket\Gamma\rrbracket\rightarrow\llbracket\Delta\rrbracket$ in $\mathbb{C}$. In particular, a term 
\[\Gamma\vdash t:E\]
is given by a section $t$ of the fibration $f^{\ast}p\colon f^{\ast}E\twoheadrightarrow\llbracket\Gamma\rrbracket$ in
$\mathbb{C}$. Under this interpretation of the structural objects, we have the following notation and correspondences of 
the logical rules in $\mathbb{C}$.
\begin{enumerate}
\item The type family obtained by substitution of a context morphism $t\colon\Gamma\rightarrow\Delta$ in a type family
$\Delta\vdash E\text{ }\mathtt{type}$ is denoted by
\[\Gamma\vdash E(t)\text{ }\mathtt{type}.\]
The interpretation $\llbracket\Gamma,e:E(t)\rrbracket$ in $\mathbb{C}$ is given by the pullback
\[\xymatrix{
\llbracket t\rrbracket^{\ast}\llbracket\Delta,e:E\rrbracket\ar[r]\ar@{->>}[d]\ar@{}[dr]|(.3){\pbs} & \llbracket\Delta,e:E\rrbracket\ar@{->>}[d] \\
\llbracket\Gamma\rrbracket\ar[r]_(.4){\llbracket t\rrbracket} & \llbracket\Delta\rrbracket.
}\]
\item The $\sum$-type family (of dependent pairs) associated to a type family
$\Gamma,b:B\vdash E\text{ }\mathtt{type}$ is denoted by
\[\Gamma\vdash\sum_{b:B}E(b)\text{ }\mathtt{type}.\]
The interpretation of the associated context extension in $\mathbb{C}$ is given by the composition of fibrations
\[\llbracket\Gamma,b:B,e:E\rrbracket\twoheadrightarrow\llbracket\Gamma,b:B\rrbracket\twoheadrightarrow\llbracket\Gamma\rrbracket.\] 
In the special case $\Gamma\vdash E\text{ }\mathtt{type}$ is a type family already, the $\sum$-type family
$\Gamma\vdash\sum_{b:B}E$ (obtained by syntactic Weakening) is simply referred to as the product type family 
of $B$ and $E$ in context $\Gamma$, and often also denoted by $B\times E$. In $\mathbb{C}$, the dependent 
sum $\llbracket\Gamma,(b,e):B\times E\rrbracket$ is a product
$\llbracket\Gamma,b:B\rrbracket\times\llbracket\Gamma,e:E\rrbracket$ in
$\mathbb{C}/\llbracket\Gamma\rrbracket$.
\item The $\prod$-type family (of dependent functions) associated to a type family as in 2.\ is denoted by
\[\Gamma\vdash\prod_{b:B}E(b)\text{ }\mathtt{type}.\]
The interpretation of its associated context extension in $\mathbb{C}$ is given by the dependent product functor 
associated to the fibration $\llbracket\Gamma,b:B\rrbracket\twoheadrightarrow \llbracket\Gamma\rrbracket$ 
applied to the fibration
$\llbracket\Gamma,b:B,e:E\rrbracket\twoheadrightarrow\llbracket\Gamma,b:B\rrbracket$.

In the special case $\Gamma\vdash E\text{ }\mathtt{type}$ is a type family already, the dependent function type family
$\Gamma\vdash\prod_{b:B}E$ (obtained by syntactic Weakening) is simply referred to as the function type 
family from $B$ to $E$ in context $\Gamma$, and often also denoted by
$\Gamma\vdash B\rightarrow E\text{ }\mathtt{type}$. In $\mathbb{C}$, the corresponding dependent product
$\llbracket \Gamma,f:B\rightarrow E\rrbracket$ is an internal hom-object for
$\llbracket\Gamma,b:B\rrbracket\twoheadrightarrow\llbracket\Gamma\rrbracket$ and
$\llbracket\Gamma,e:E\rrbracket\twoheadrightarrow \llbracket\Gamma\rrbracket$ in
$\mathbb{C}/\llbracket\Gamma\rrbracket$, and hence often also denoted as
$[\llbracket\Gamma,b:B\rrbracket,\llbracket\Gamma,e:E\rrbracket]_{\mathbb{C}/\llbracket\Gamma\rrbracket}$.
\item The identity type family associated to a type family $\Gamma\vdash E\text{ }\mathtt{type}$ is denoted by
\[\Gamma,e_1:E,e_2:E\vdash e_1=_{E} e_2\text{ }\mathtt{type}.\]
The interpretation of its associated context extension is given by a suitably chosen path-object fibration
$P_{\llbracket\Gamma\rrbracket}(\llbracket\Gamma,e:E\rrbracket)\twoheadrightarrow \llbracket\Gamma,e:E\rrbracket\times_{\llbracket\Gamma\rrbracket} \llbracket\Gamma,e:E\rrbracket$ in $\mathbb{C}$.
\end{enumerate}

From these type formers, we may build the compound type family of equivalences 
associated to a type family $\Gamma,b:B\vdash E\text{ }\mathtt{type}$ as follows. A function
$\Gamma\vdash f\colon E(a)\rightarrow E(b)$ for terms $\Gamma\vdash a:B$, $\Gamma\vdash b:B$ is left (right) invertible if there is a 
term of type
\begin{align*}
\Gamma\vdash\mathrm{Linv}(f):\equiv\sum_{g:E(b)\rightarrow E(a)}\big(gf=_{E(a)\rightarrow E(a)}1_{E(a)}\big)
\text{ }\mathtt{type},\\
\Gamma\vdash\mathrm{Rinv}(f):\equiv\sum_{h:E(b)\rightarrow E(a)}\big(fh=_{E(b)\rightarrow E(b)}1_{E(b)}\big)
\text{ }\mathtt{type},
\end{align*}
respectively. Here, the composition $gf$ is defined by a basic $\lambda$-abstraction of the form $\lambda e.g(f(e))$
(\cite[Section 1.2]{hott}). The identity $1_{E(b)}$ is given by a canonical function term of the form $\lambda e.e$.
Accordingly, a function $\Gamma\vdash f\colon E(a)\rightarrow E(b)$ is an equivalence if there is a term of type
\[\Gamma\vdash\mathrm{Biinv}(f):\equiv\mathrm{Linv}(f)\times\mathrm{Rinv}(f)\text{ }\mathtt{type}.\]
The type family of equivalences associated to $\Gamma,b:B\vdash E\text{ }\mathtt{type}$ is thus given by
\begin{align}\label{equequivtypesyn}
\Gamma,a:B,b:B\vdash\mathrm{Equiv}(E):\equiv\sum_{f:E(a)\rightarrow E(b)}\mathrm{Biinv}(f)\text{ }\mathtt{type}.
\end{align}
For any such type family $\Gamma,b:B\vdash E\text{ }\mathtt{type}$, there is a canonical function
\begin{align}\label{equidtoequiv}
\Gamma,a:B,b:B\vdash \mathrm{idtoequiv}_E:\left(a=_B b\rightarrow\mathrm{Equiv}(E)(a,b)\right)
\end{align}
obtained by \emph{path induction} (\cite[Section 1.12]{hott}, classically referred to as J-elimination) applied to the function
$\Gamma\vdash\lambda(a,a,\mathrm{refl}_a).w(1_{E(a)}):\mathrm{Equiv}(E)(a,a)$, where $w(1_{E(a)})$ is the canonical 
proof that the identity is an equivalence. 
\begin{definition}[{\cite[Definition 3.1.3]{klvsimp}}]\label{defunivalencesyn}
A type family $\Gamma,b:B\vdash E\text{ }\mathtt{type}$ is \emph{univalent} if the function (\ref{equidtoequiv}) itself 
is an equivalence. That is, if the type family
\[\Gamma,a:B,b:B\vdash\mathrm{Biinv}(\mathrm{idtoequiv}_E(a,b))\text{ }\mathtt{type}\]
is inhabited.
\end{definition}
Univalence in the sense of Definition~\ref{defunivalencesyn} of the Tarski universal type families
$\Gamma\vdash\mathcal{U}\text{ }\mathtt{type}$  is exactly Voevodsky's Univalence Axiom (\cite[Section 2.10]{hott}) whenever they
exist.
\begin{remark}\label{remsyntosemequiv}
Given a map $f\colon E_1\rightarrow E_2$ between fibrations over an object $B$ in $\mathbb{C}$, the associated 
type family $b:B\vdash \mathrm{Biinv}(f(b))\text{ }\mathtt{type}$ is inhabited if and only if the map $f$ is a homotopy-equivalence in 
$\mathbb{C}$ over $B$ (see \cite[Section 5]{shulmaninv}).
\end{remark}
This subsumes the essential corner points of the syntax and semantics associated to a type theoretic fibration category 
$\mathbb{C}$ which underlie the constructions in this section. The rest of this paper is largely syntax-free. The 
compound categorical interpretation of the type of equivalences (\ref{equequivtypesyn}) associated to a type family in
$\mathcal{T}_{\mathbb{C}}$, as well as the interpretation of univalence correspondingly, will be subject matter of the 
rest of this section, embedded in a more general treatment of Segal objects in $\mathbb{C}$ and the type family of 
internal equivalences thereof.

\paragraph{Segal objects in a type theoretic fibration category}
For natural numbers $i<n$ let $j^i$ be the ``$i$-th essential edge'' of $[n]$ in the simplex category $\Delta$. That 
means, $j^i\colon[1]\rightarrow[n]$ is the morphism given by $k\mapsto k+i$.
Let $s\mathbb{C}=\mathbb{C}^{\Delta^{op}}$ be the category of simplicial objects in $\mathbb{C}$.
Then, given a simplicial object $X$ in $\mathbb{C}$ for which the 1-boundaries $d_i\colon X_1\rightarrow X_0$,
$i\in\{0,1\}$, are fibrations, the $n$-th \emph{Segal map}
\[\xi_n\colon X_n\rightarrow X_1\times_{X_0}\dots\times_{X_0}X_1\]
associated to $X$ is the natural map associated to the cone
\[\xymatrix{
& & & X_n\ar@/_1pc/[dlll]_{(j^0)^{\ast}}\ar@/_/[dl]^{(j_1)^{\ast}} \ar@/^1pc/[drrr]^{(j^{n-1})^{\ast}}\ar@/^/@{-->}[dr] & & & \\
X_1\ar[dr]_{d_0} & &  X_1\ar[dl]^{d_1}\ar[dr]_{d_0} & & \dots\ar@{-->}[dl]^{d_1}\ar@{-->}[dr]_{d_0} & & X_1.\ar[dl]^{d_1}\\
 & X_0 & & X_0 & & X_0 & 
 }\]

\begin{definition}\label{defsegalob}
Let $X$ be a simplicial object in $\mathbb{C}$.
\begin{enumerate}
\item $X$ is \emph{sufficiently fibrant} if both the 2-Segal map $\xi_2\colon X_2\rightarrow X_1\times_{X_0}X_1$
and the boundary map $(d_1,d_0)\colon X_1\rightarrow X_0\times X_0$
are fibrations in $\mathbb{C}$.
\item Let $X$ be sufficiently fibrant. We say that $X$ is a \emph{Segal object} (\emph{strict Segal object}) if the 
associated Segal maps $\xi_n\colon X_n\rightarrow X_1\times_{X_0}\dots\times_{X_0}X_1$ are homotopy-equivalences 
(isomorphisms) in $\mathbb{C}$.
\end{enumerate}
\end{definition}

In contrast to the common definitions of Segal objects in the literature (e.g.\ \cite{brncatcomp}), we 
purposely do not assume Reedy fibrancy in the definition of Segal objects here. Therefore, for instance in the case when
$\mathbb{C}$ is the category of Kan complexes and Kan fibrations, Segal spaces are exactly the Reedy fibrant Segal 
objects.

\begin{remark}\label{remsegalhmtpy}
Whenever a simplicial object $X$ is sufficiently fibrant, the ordinary pullbacks
$X_1\times_{X_0}\dots\times_{X_0} X_1$ in the codomain of the Segal maps are homotopy-pullbacks. Hence, being a Segal object is a 
homotopy-invariant property among sufficiently fibrant simplicial objects.
\end{remark}

Thus, a sufficiently fibrant simplicial object $X$ in $\mathbb{C}$ gives rise to a type family $X_1(x,y)$ of edges 
fibred over pairs of objects $x,y$ in $X$, and a type $X_2(f,g)$ of fillers fibred over pairs of concatenable 
edges $f,g$ in $X_1$. The definition of sufficient fibrancy is chosen in such a way that, first, all properties 
considered in this paper are stable under homotopy-equivalence between sufficiently fibrant objects, and second, the 
following example -- which in general is not Reedy fibrant (see Remark~\ref{remnervenotreedy}) -- is sufficiently 
fibrant as well.
\paragraph*{The nerve of a fibration}
Let $p\colon E\twoheadrightarrow B$ be a fibration in $\mathbb{C}$ and denote the internal hom-object
$[\pi_1^{\ast}E,\pi_2^{\ast}E]_{B\times B}$ by $\mathrm{Fun}(p)$. This is the type of fiberwise functions associated to 
$p$. It comes together with a source and target fibration 
\begin{align}\label{equnervegraph}
\mathrm{Fun}(p)\overset{(s,t)}{\twoheadrightarrow} B\times B
\end{align}
and a reflexivity map $r\colon B\rightarrow\mathrm{Fun}(p)$ representing the identity on $E$. In particular, the triple 
$(s,t,r)$ forms a reflexive graph in $\mathbb{C}$.

In terms of the internal type theory of $\mathbb{C}$, the type $\mathrm{Fun}(p)$ corresponds to the parametrized function 
type
\begin{align}\label{equfcttype}
\vdash\sum_{a,b:B}(E(a)\rightarrow E(b))\text{ }\texttt{type}
\end{align}
with corresponding source and target maps specified by the terms $\lambda (a,b,f).a$ and $\lambda (a,b,f).b$. The 
reflexivity map is specified by $\lambda b.(b,b,\mathrm{id}_{E(b)})$.

Recall, e.g.\ from \cite[3]{horel}, that the category $\mathbb{C}/(B\times B)$ of \emph{internal graphs} with vertices 
in $B$ has a tensor product $\otimes_B$ given by pullback
\[\xymatrix{
(E\xrightarrow{(s,t)}B^2)\otimes_B(E\sprime\xrightarrow{(s\sprime,t\sprime)}B^2)\ar[r]\ar[d]\ar@{}[dr]|(.3){\pbs} & E\sprime\ar[d]_{s\sprime}\ar@/^/[ddr]^{t\sprime} & \\
E\ar@/_/[drr]_{s}\ar[r]_{t} & B & \\
 & & B. \\
}\]
The monoids in the monoidal category $(\mathbb{C}/(B\times B),\otimes_B)$ are the internal categories in 
$\mathbb{C}$ over $B$, the category of these monoids is denoted by $\mathrm{ICat}(\mathbb{C})_B$. The assignment 
$B\mapsto\mathrm{ICat}(\mathbb{C})_B$ induces a pseudo-functor
$\mathrm{ICat}\colon\mathbb{C}^{op}\rightarrow\mathbf{Cat}$. The Grothendieck construction applied to this
pseudo-functor is the category $\mathrm{ICat}(\mathbb{C})$ of internal categories in $\mathbb{C}$. The following 
proposition is a straightforward generalisation of \cite[Example 7.1.4.(ii)]{jacobsttbook} to type theoretic fibration 
categories. Although the author assumes local cartesian closedness in the cited example, his construction can be carried 
out for every fibration in a type theoretic fibration category by Condition 5.\ in Definition~\ref{defttfibcat}. 

\begin{proposition}\label{propfunicat}
For a fibration $p\colon E\twoheadrightarrow B$ in $\mathbb{C}$, the graph
\[\xymatrix{\mathrm{Fun}(p)\ar@<.5ex>[r]^{s}\ar@<-.5ex>[r]_{t} &B}\]
comes equipped with a multiplication that turns this graph into an internal category object in $\mathbb{C}$ with its 
unit given by the reflexivity $r$. 
\end{proposition}\qed

\begin{proposition}\label{propnervesegal}
There is a nerve functor $N\colon\mathrm{ICat}(\mathbb{C})\rightarrow s\mathbb{C}$ that, when restricted to its image, 
yields an equivalence to the subcategory of objects in $s\mathbb{C}$ whose Segal maps are isomorphisms.
\end{proposition}

\begin{proof}
The functor $N$ has a straightforward definition, given as the reflexive free category comonad resolution in
$\mathbb{C}$. That means, given an object $C\in\mbox{ICat}(\mathbb{C})$ with underlying reflexive graph
\begin{equation}\label{diagreflgraph}
\xymatrix{C_1\ar@<1ex>[r]^{s}\ar@<-1ex>[r]_{t} & C_0\ar[l]|{\eta}},
\end{equation}
we define 
\[N(C)_n:=\bigotimes^{n}_{C_0}(\xymatrix{C_1\ar@<-.5ex>[r]_t\ar@<.5ex>[r]^s & C_0}).\]
Here, the right hand side is the $n$-fold monoidal power of the underlying graph in $\mathbb{C}/(C_0\times C_0)$, given 
by the $n$-fold pullback $C_1\times_{C_0}\dots\times_{C_0} C_1$. The degeneracy and two boundaries between $N(C)_0$ 
and $N(C)_1$ are exactly given by the data in (\ref{diagreflgraph}).

Higher degeneracies are given by inserting the unit $\eta$ into corresponding components of the monoidal 
power and the boundaries are given by the obvious combination of multiplications and projections. One can 
verify the simplicial relations one by one and observe that they hold by the associativity and unitality laws 
satisfied by the multiplication $\mu$ and the unit $\eta$. This defines a simplicial object $N(C)$ in $\mathbb{C}$.

Given a simplicial object $X$ in $\mathbb{C}$ whose Segal maps are isomorphisms, one easily computes that its
$2$-skeleton yields an object $X_{\leq 2}\in\mathrm{ICat}(\mathbb{C})_{X_0}$ such that $N(X_{\leq 2})\cong X$.
\end{proof}

\begin{remark}
In the case $\mathbb{C}$ is the fibration category of Kan complexes, the nerve functor from 
Proposition~\ref{propnervesegal} is shown in \cite{horel} to be part of an equivalence between a suitably defined 
homotopy theory on $\mathrm{ICat}(\mathbb{C}$), and the homotopy theory of complete Segal spaces. Internal category 
objects in Kan complexes hence yield another model for $(\infty,1)$-category theory.
\end{remark}

By Propositions~\ref{propfunicat} and \ref{propnervesegal} we obtain for every fibration
$p\colon E\twoheadrightarrow B$ in $\mathbb{C}$ the simplicial object
\begin{align}\label{diagdefNp}
N(p):=N\big(\xymatrix{\mathrm{Fun}(p)\ar@<.5ex>@/^/[r]^s\ar@<-.5ex>@/_/[r]_t & B\ar[l]|{r}\big)}\in s\mathbb{C}.
\end{align}
The nerve $N(p)$ is sufficiently fibrant, because the boundary
\[(d_1,d_0)\colon N(p)_1\rightarrow N(p)_0\times N(p)_0\]
is the fibration $(s,t)\colon\mathrm{Fun}(p)\twoheadrightarrow B\times B$, and the 2-Segal map $\xi_2$ is an isomorphism 
(and in particular a fibration). Therefore, the simplicial object $N(p)$ is a strict Segal object in $\mathbb{C}$ for 
every fibration $p\in\mathbb{C}$, again by Proposition~\ref{propnervesegal}. 

\begin{remark}\label{remnervenotreedy}
The nerve $N(p)$ of a fibration $p$ in a type theoretic fibration category $\mathbb{C}$ is generally not Reedy fibrant. 
For instance, consider the matching map $N(p)_2\rightarrow M_2(N(p))$, which, by definition, is isomorphic to the 
boundary map
\begin{align}\label{equremnervenotreedy}
(\pi_1,\pi_2,\circ)\colon\mathrm{Fun}(p)\otimes_B\mathrm{Fun}(p)\rightarrow(\mathrm{Fun}(p)\otimes_B\mathrm{Fun}(p))\times_{B\times B}\mathrm{Fun}(p).
\end{align}
Here, $\circ\colon\mathrm{Fun}(p)\otimes_B\mathrm{Fun}(p)\rightarrow\mathrm{Fun}(p)$ denotes the monoidal 
multiplication associated to the internal category object $\mathrm{Fun}(p)$. The map (\ref{equremnervenotreedy}) is a 
fibration if and only if it has the right lifting property with respect to all acyclic cofibrations in $\mathcal{C}$. 
Unfolding the definitions, a respective lifting problem corresponds to a diagram of the form
\begin{align}\label{diagremnervenotreedy}
\begin{gathered}
\xymatrix{
    &     &     &   & Y\ar@{->>}[dd]|\hole\ar[dr] &   \\
    & Y_1\ar@{->>}[dd]|\hole\ar[dr]\ar[urrr] &     & X\ar@/_/@{->>}[dr]|(.45)\hole\ar[rr]\ar[ur] &   & Z\ar@/^/@{->>}[dl] \\
X_1\ar@/_/@{->>}[dr]\ar[rr]\ar[ur]\ar[urrr]|(.34)\hole|\hole &     & Z_1\ar@/^/@{->>}[dl]\ar[urrr] &   & C &   \\
    & A\ar@/_/@{^(->}[urrr]^{\sim}   &     &   &   &  
}
\end{gathered}
\end{align}
where $A\overset{\sim}{\hookrightarrow}C$ is an acyclic cofibration in $\mathcal{C}$, all surfaces but the triangle of 
fibrations over $C$ commute (strictly), the vertical fibrations are pullbacks of $p$, and all five diagonally horizontal 
squares over the base $A\rightarrow C$ are pullback squares. The lifting problem associated to a diagram of the form 
(\ref{diagremnervenotreedy}) has a solution if and only if the triangle over $C$ can be shown to (strictly) commute as 
well.

In other words, this means that in order to show that a triangle of pullbacks of $p$ with base $C$ commutes, it 
suffices to show that its restriction along an acyclic cofibration with codomain $C$ commutes. Indeed, solutions to the 
associated lifting problem are unique whenever they exist because the map (\ref{equremnervenotreedy}) is a monomorphism.

To give an example of a fibration $p$ such that the matching map (\ref{equremnervenotreedy}) is not a fibration, and so 
$N(p)$ is not Reedy fibrant, let $\mathbb{C}$ be the category of Kan complexes. Let $p$ be the
Kan fibration $\pi_{\kappa}\colon\tilde{U}_{\kappa}\twoheadrightarrow U_{\kappa}$ universal for $\kappa$-small Kan 
fibrations for some regular cardinal $\kappa$ (\cite{klvsimp}). Let $J=N(I[1])$ be the walking isomorphism (see
Section~\ref{seccompleteness}) and $1\colon J\rightarrow J$ the map with constant value $1\in J$. Then the triangle of fibrations
\[\xymatrix{
 & J\times J\ar@{->>}[dd]|(.4)\hole|(.45)\hole|\hole^(.25){\pi_2}\ar[dr]^{(\mathrm{id},\pi_2)} &  \\
 J\times J\ar@/_/@{->>}[dr]_(.45){\pi_2}\ar[rr]^(.4){(((\mathrm{id},1)\pi_1),\pi_2)}\ar@{=}[ur] & & (J\times J)\times J\ar@/^/@{->>}[dl]^{\pi_2} \\
& J &   \\
}\]
does not commute strictly (as can be verified on the vertex $(0,0)\in J\times J$), but restricts along the acyclic 
cofibration $\{1\}\colon\Delta^0\hookrightarrow J$ to the commutative triangle
\[\xymatrix{
 & J\ar@{->>}[dd]|\hole\ar[dr]^(.4){(\mathrm{id},1)} & \\
J\ar@/_/@{->>}[dr]\ar[rr]^(.3){(\mathrm{id},1)}\ar@{=}[ur] &  & J\times J.\ar@/^/@{->>}[dl] \\   
 & \Delta^0 &     
}\]
Thus, we obtain a lifting problem for (\ref{equremnervenotreedy}) in $\mathbb{C}$ that has no solution. 
\end{remark}

\paragraph*{The type of internal equivalences}
Following \cite[Section 5]{shulmaninv}, for a fibration $p\colon E\twoheadrightarrow B$ one can construct the
fibration of fiberwise homotopy-equivalences $\mathrm{Equiv}(p)\twoheadrightarrow\mathrm{Fun}(p)\twoheadrightarrow B\times B$ which 
computes the interpretation of the type family of equivalences associated to $b:B\vdash E(b)\text{ }\mathtt{type}$ as defined in 
(\ref{equequivtypesyn}), and show that this type family is univalent in the sense of Definition~\ref{defunivalencesyn} 
if and only if a canonically associated reflexivity map $r\colon B\rightarrow\mathrm{Equiv}(p)$ is acyclic (see 
\cite[Sections 5 and 7]{shulmaninv}). 
But, in fact, we can define a fibration of internal equivalences $\mathrm{Equiv}(X)$ for every sufficiently fibrant 
simplicial object $X$ in $\mathbb{C}$ and hence define a notion of univalence in this generality. In the spirit of 
Joyal's definition of biinvertible functions in type theory -- i.e.\ terms of the type in (\ref{equequivtypesyn}) -- the simplicial 
structure of $X$ suggests the definition of a type family of equivalences associated to $X$ in the underlying syntax as follows.

Let $X$ be a sufficiently fibrant simplicial object in $\mathbb{C}$. In type theoretic terms, via the second Segal map 
we have a context
\[x,y,z:X_0,f:X_1(x,y),g:X_1(y,z)\vdash X_2(f,g)\text{ }\texttt{type}.\]
Then we may define a type family $x,y:X_0\vdash \mathrm{Equiv}(X)$ by $\sum$-formation of
\begin{align}\label{equequivtype}
x,y:X_0,f:X_1(x,y)\vdash\mathrm{Equiv}(x,y,f):\equiv \mathrm{Linv}(x,y,f)\times\mathrm{Rinv}(x,y,f)
\end{align}
for
\begin{align}\label{equequivtypepartials}
\mathrm{Linv}(x,y,f) & :\equiv\sum_{g:X_1(y,x)}\sum_{\sigma:X_2(f,g)} d_1\sigma=_{X_1(x,x)}s_0 x, \\
\notag \mathrm{Rinv}(x,y,f) & :\equiv\sum_{h:X_1(y,x)}\sum_{\sigma:X_2(h,f)} d_1\sigma=_{X_1(y,y)}s_0 y.
\end{align}

\begin{remark}
If $X\in s\mathbb{C}$ is a Segal object, the Segal map $\xi_2\colon X_2\twoheadrightarrow X_1\times_{X_0} X_1$ exhibits a 
(homotopically unique) section $\zeta$ with induced composition $\kappa$ 
given by $d_1\zeta\colon X_1\times_{X_0}X_1\rightarrow X_1$. We thereby obtain an equivalence of the type
$\mathrm{Equiv}X(x,y)$ to the type
\begin{align}
\sum_{f:X_1(x,y)}\Big(\sum_{g:X_1(y,x)}\big(\kappa(f,g)=_{X_1(x,x)}s_0 x\big)\times\sum_{h:X_1(y,x)}\big(\kappa(h,f)=_{X_1(y,y)}s_0 y\big)\Big)
\end{align}
in context $x,y:X_0$.
\end{remark}
Just as the map in (\ref{equidtoequiv}), by path induction on the trivial equivalence
\[x:X_0\vdash \mathrm{deg}_x:\equiv(s_0x,s_0x,s_0s_0x,\mathrm{refl}_{s_0x}),(s_0x,s_0x,s_0s_0x,\mathrm{refl}_{s_0x}):\mathrm{Equiv}(X)(x,x)\]
we obtain a canonical map of the form
\begin{align}\label{equdefunivalencesegalsyn}
x,y:X_0\vdash\mathrm{idtoequiv}_X:(x=_{X_0}y\rightarrow\mathrm{Equiv}(X)).
\end{align}

\begin{definition}\label{defequivXalt}
A sufficiently fibrant Segal object $X\in\mathbb{C}$ is \emph{univalent} if the map (\ref{equdefunivalencesegalsyn}) is an equivalence 
in context $x,y:X_0$.
\end{definition}

\begin{lemma}\label{lemmasimpunivalence}
Let $p\colon E\twoheadrightarrow B$ be a fibration in $\mathbb{C}$. Then $p$ is a univalent fibration in $\mathbb{C}$ in 
the sense of Definition~\ref{defunivalencesyn} if and only if the Segal object $N(p)$ is univalent in the sense of 
Definition~\ref{defequivXalt}.
\end{lemma}
\begin{proof}
If we substitute the Segal object $N(p)$ for $X$ in (\ref{equequivtype}), via (\ref{equfcttype}) we see that the object
$\mathrm{Equiv}(N(p))$ as a type family is given up to judgemental equality in context $a,b:B$ by
\[\sum_{f:E(a)\rightarrow E(b)}\Big(\sum_{g:E(b)\rightarrow E(a)}\big(gf=_{E(a)\rightarrow E(a)}1_{E(a)}\big)\times\sum_{h:E(b)\rightarrow E(a)}\big(fh=_{E(b)\rightarrow E(b)}1_{E(b)}\big)\Big).\]
This is exactly the type family $\mathrm{Equiv}(E)$ of equivalences associated to $p$ as a type family from 
(\ref{equequivtypesyn}). As in particular $N(p)_0=B$, we see that the type judgements (\ref{equidtoequiv}) and 
(\ref{equdefunivalencesegalsyn}) coincide for $X=N(p)$. It follows that the Definitions \ref{defunivalencesyn} and 
\ref{defequivXalt} of univalence coincide in this case.
\end{proof}

In closing, we construct a categorical interpretation of the type family $\mathrm{Equiv}(X)$ for a sufficiently fibrant simplicial 
object $X$ in $\mathbb{C}$. Thus, in order to define the fibrations of left- and right-invertible maps $\mathrm{Linv}(X)$ and
$\mathrm{Rinv}(X)$ over $X_1$ respectively, we construct the object $\mathrm{Inv}(X)$ of triangles in $X_2$ whose
1-boundary is degenerate up to homotopy in $X_1$. It will serve as domain for both fibrations. Therefore, 
consider the pullbacks
\begin{align}\label{diagconvpair}
\begin{gathered}
\xymatrix{
\Delta^{\ast}X_2\ar[r]\ar@{}[dr]|(.3){\pbs}\ar@{->>}[d]_{\Delta^{\ast}d_1} & X_2\ar@{->>}[d]^{d_1}\\
\Delta^{\ast}X_1\ar@{->>}[d]_{\Delta^{\ast}(d_1,d_0)}\ar[r]^{} \ar@{}[dr]|(.3){\pbs}& X_1\ar@{->>}[d]^{(d_1,d_0)} \\
X_0\ar[r]_{\Delta} & X_0\times X_0.
}
\end{gathered}
\end{align}
The 1-boundary $d_1\colon X_2\twoheadrightarrow X_1$ and its degenerate postcomposition
$s\colon X_2\xrightarrow{d_1}X_1\xrightarrow{d_1} X_0\xrightarrow{s_0} X_1$ 
can be reindexed to a pair
\[(\Delta^{\ast}d_1,\Delta^{\ast}s)\colon\Delta^{\ast}X_2\rightarrow\Delta^{\ast}X_1\times_{X_0}\Delta^{\ast}X_1\]
which takes a triangle of the form
\begin{align}\label{diagendotriangles}
\begin{gathered}
\xymatrix{
 & y\ar@/^/[dr]^g & \\
 x\ar@/^/[ur]^f\ar@/_/[rr]_h\ar@{}[urr]_{\sigma} & & x 
}
\end{gathered}
\end{align}
to the pair $(h,s_0x)$ of parallel edges. Rather than equalizing these two maps on the nose, 
the object $\mathrm{Inv}(X)$ is defined as the object of homotopies in $X_1$ between the two maps, defined by the 
following pullback.
\begin{align}\label{diageqdj}
\begin{gathered}
\xymatrix{
\mathrm{Equalizer}(\Delta^{\ast}d_1,\Delta^{\ast}s)\ar@{}[dr]|(.3){\pbs}\ar[r]\ar[d] & \mathrm{Inv}(X)\ar@{}[dr]|(.3){\pbs}\ar@{->>}[r]\ar[d] & \Delta^{\ast}X_2\ar[d]^{(\Delta^{\ast}d_1,\Delta^{\ast}s)}\\
\Delta^{\ast}X_1\ar@{^(->}[r]^(.45){\sim}\ar@/_2pc/[rr]_{\mathrm{diag}} & P_{X_0}(\Delta^{\ast}X_1)\ar@{->>}[r] & \Delta^{\ast}X_1\times_{X_0}\Delta^{\ast}X_1
}
\end{gathered}
\end{align}
Here, the homotopy type of $\mathrm{Inv}(X)$ does not depend on any given choice of path-object
$P_{X_0}(\Delta^{\ast}X_1)$, since homotopy-equivalences in $\mathbb{C}$ between fibrations are pullback-stable.

The outer boundary maps $d_i\colon X_2\rightarrow X_1$ ($i=0,2$) factor through the 2-Segal map of $X$, and hence are 
fibrations since $X$ is sufficiently fibrant. We thus may define composite fibrations
\[\mathrm{Linv}(X)\colon\mathrm{Inv}(X)\twoheadrightarrow\Delta^{\ast}X_2\overset{d_2}{\twoheadrightarrow} X_1,\]
\[\mathrm{Rinv}(X)\colon\mathrm{Inv}(X)\twoheadrightarrow\Delta^{\ast}X_2\overset{d_0}{\twoheadrightarrow} X_1,\]
and
\begin{align}\label{defequtt}
\mathrm{Equiv}(X)\colon\mathrm{Linv}X\times_{X_1}\mathrm{Rinv}X\twoheadrightarrow X_1
\end{align}
accordingly. The degeneracy $s_0\colon X_0\rightarrow X_1$, via the composition $s_0s_0\colon X_0\rightarrow X_2$, factors 
through the strict equalizer in Diagram~(\ref{diageqdj}). We obtain a canonical lift of the form
\begin{align}\label{diagequivdeg}
\begin{gathered}
\xymatrix{
 & \mathrm{Equiv}(X)\ar@{->>}[d] \\
 X_0\ar[r]_{s_0}\ar@{-->}[ur]^{s_0} & X_1.
}
\end{gathered}
\end{align}

\begin{proposition}\label{defunivalenceso}
Let $X$ be a sufficiently fibrant simplicial object in $\mathbb{C}$. 
\begin{enumerate}
\item There is a homotopy-equivalence
\begin{align}\label{defunivalencesop1}
\begin{gathered}
\xymatrix{
X_0\ar[r]^{s_0}\ar[d]_{\llbracket \mathrm{deg}_x\rrbracket} & \mathrm{Equiv}(X)\ar@{->>}[d] \\
\llbracket x,y:X_0,w:\mathrm{Equiv}(X)\rrbracket\ar@{->>}[r]\ar[ur]^{\simeq} & X_1 
}
\end{gathered}
\end{align}
in $\mathbb{C}$ such that both triangles commute.
\item The simplicial object $X$ is univalent in the sense of Definition~\ref{defequivXalt} if and only if the canonical lift in 
(\ref{diagequivdeg}) is a homotopy-equivalence in $\mathbb{C}$.
\end{enumerate}
\end{proposition}
\begin{proof}
Following the template to interpret compound type formers as outlined in the beginning of this section, the context
$\llbracket x,y:X_0,w:\mathrm{Equiv}(X)\rrbracket$ in $\mathbb{C}$ is given by the following sequence of recursions. First, we have
a pullback square in $\mathbb{C}$ of the form
\begin{align}\label{diagdefunivalenceso1}
\begin{gathered}
\xymatrix{
\llbracket x,y:X_0,f:X_1(x,y),g:X_1(y,x)\rrbracket\ar@{->>}[r]^(.8){}\ar@{->>}[d]_{}\ar@{}[dr]|(.3){\pbs} & X_1\ar@{->>}[d]^{(d_1,d_0)} \\
X_1\ar@{->>}[r]_{(d_0,d_1)} & X_0\times X_0.
}
\end{gathered}
\end{align}
The two projections $\pi_i\colon X_0\times X_0\twoheadrightarrow X_0$ induce forgetful maps
$F_i\colon X_1\times_{X_0\times X_0}X_1\rightarrow X_1\times_{X_0}X_1$ which fit into pullback squares
\[\xymatrix{
X_1\times_{(X_0\times X_0)}\times X_1\ar[r]^(.55){F_i}\ar@{->>}[d]_{}\ar@{}[dr]|(.3){\pbs} & X_1\times_{X_0} X_1\ar@{->>}[d]^{((d_1,d_0)\pi_1,d_0\pi_2))} \\
X_0\times X_0\ar[r]_(.4){} & X_0\times X_0\times X_0,
}\]
where the bottom map is $\mathrm{id}\times \pi_1$ in the case $i=1$, and $\pi_2\times\mathrm{id}$ in the case $i=2$.
Pullback of the 2-Segal map $\xi_2\colon X_2\twoheadrightarrow X_1\times_{X_0}X_1$ along these two forgetful maps 
$F_i$ yields the fibrations
\[\llbracket x,y:X_0,f:X_1(x,y),g:X_1(y,x),\sigma:X_2(f,g)\rrbracket\twoheadrightarrow\llbracket x,y:X_0,f:X_1(x,y),g:X_1(y,x)\rrbracket\]
and
\[\llbracket x,y:X_0,f:X_1(x,y),g:X_1(y,x),\sigma:X_2(g,f)\rrbracket\twoheadrightarrow\llbracket x,y:X_0,f:X_1(x,y),g:X_1(y,x)\rrbracket,\]
respectively. A direct comparison of universal properties shows that these two pullbacks $F_i^{\ast}X_2$ are canonically 
isomorphic to the pullback $\Delta^{\ast}X_2$ in (\ref{diagconvpair}), and that their postcomposition with the left 
vertical projection in (\ref{diagdefunivalenceso1}) are isomorphic to the fibrations
\[d_2\colon\Delta^{\ast}X_2\twoheadrightarrow X_1\]
and
\[d_0\colon\Delta^{\ast}X_2\twoheadrightarrow X_1,\]
respectively. Under these isomorphisms, the interpretations $\llbracket d_1\sigma\rrbracket$ and $\llbracket s_0 x\rrbracket$ of the 
two terms occurring in the identity type families in (\ref{equequivtypepartials}) are exactly the right vertical maps in 
Diagram~(\ref{diageqdj}). Up to isomorphism, by the interpretation of identity types as path-objects
(\cite[Section 3.4.3]{lwlocaluniv}) and the interpretation of substitutions as pullbacks, we thus obtain a pullback square of the form
\[\xymatrix{
& \llbracket x,y:X_0,f:X_1(x,y),w_1:\mathrm{Linv}(f)\rrbracket\ar@{->>}[r]\ar[d]\ar@{}[dr]|{\pbs} & \Delta^{\ast}X_2\ar[d]^{(\Delta^{\ast}d_1,\Delta^{\ast}s)}\\
\Delta^{\ast}X_1\ar@{^(->}[r]^(.2){\sim} & \llbracket x:X_0,f:X_1(x,x),g:X_1(x,x),w:f=_{X_1(x,x)}g\rrbracket\ar@{->>}[r]  & \Delta^{\ast}X_1\times_{X_0}\Delta^{\ast}X_1. 
}\]
This is exactly Diagram~(\ref{diageqdj}) up to the explicit choice of a path-object for $\Delta^{\ast}X_1$ in $\mathbb{C}$. We hence 
obtain a homotopy-equivalence between the pullbacks $\llbracket x,y:X_0,f:X_1(x,y),w_1:\mathrm{Linv}(f)\rrbracket$ and
$\mathrm{Inv}(X)$ over $\Delta^{\ast}X_2$. In particular, by postcomposition of the diagram above with
$d_2\colon\Delta^{\ast}X_2\twoheadrightarrow X_1$ on the upper right, and 
precomposition with $s_0\colon X_0\rightarrow\Delta^{\ast}X_1$ on the lower left, we thus obtain a homotopy-equivalence
\[\xymatrix{
X_0\ar[r]^{s_0}\ar[d]_{\llbracket (s_0x,s_0x,s_0s_0x,\mathrm{refl}{s_0x})\rrbracket} & \mathrm{Inv}(X)\ar@{->>}[d]^{\mathrm{Linv}(X)} \\
\llbracket x,y:X_0,f:X_1(x,y),w_1:\mathrm{Linv}(f)\rrbracket\ar@{->>}[r]\ar[ur]^{\simeq} & X_1
}\]
making both triangles commute. Via the isomorphism
\[\xymatrix{
\llbracket x,y:X_0,f:X_1(x,y),(g,\sigma,w):\mathrm{Linv}(f)\rrbracket\ar[dd]_{\llbracket(y,x,g,f,\sigma,w)\rrbracket}^{\cong}\ar@{->>}@/^1pc/[dr]^(.8){d_2} & \\
 & \Delta^{\ast}X_1 \\
\llbracket x,y:X_0,f:X_1(x,y),(g,\sigma,w):\mathrm{Rinv}(f)\rrbracket\ar@{->>}@/_1pc/[ur]_(.8){d_0} &
}\]
over $\Delta^{\ast}X_1 $, we also obtain a homotopy equivalence
\[\xymatrix{
X_0\ar[r]^{s_0}\ar[d]_{\llbracket (s_0x,s_0x,s_0s_0x,\mathrm{refl}{s_0x})\rrbracket} & \mathrm{Inv}(X)\ar@{->>}[d]^{\mathrm{Rinv}(X)} \\
\llbracket x,y:X_0,f:X_1(x,y),w_1:\mathrm{Rinv}(f)\rrbracket\ar@{->>}[r]\ar[ur]^{\simeq} & X_1
}\]
making both triangles commute. Eventually, we obtain a homotopy-equivalence as in (\ref{defunivalencesop1}) by (\ref{equequivtype}) and (\ref{defequtt}). This 
finishes Part 1. 

For Part 2., it follows again from \cite[Section 3.4.3]{lwlocaluniv} and Remark~\ref{remsyntosemequiv} that $X$ is univalent if and only if
the (homotopically unique) diagonal lift
\[\xymatrix{
X_0\ar@{^(->}[d]_{\rotatebox[origin=c]{90}{$\sim$}}\ar[r]^(.3){\llbracket\mathrm{deg}_{x}\rrbracket} & \llbracket\mathrm{Equiv}(X)\rrbracket\ar@{->>}[d] \\
P X_0\ar@{->>}[r]\ar@{-->}[ur] & X_0\times X_0
}\]
is a homotopy-equivalence. Hence, by 2-out-of-3 and 
Part 1., this holds if and only if the lifted degeneracy $s_0\colon X_0\rightarrow\mathrm{Equiv}(X)$ is a homotopy-equivalence.
\end{proof}

\begin{remark}
For $X$ the nerve of a fibration $p$ in the fibration category $\mathbb{C}$ of Kan complexes,
Proposition~\ref{defunivalenceso} is closely related to the correspondence between ``type theoretic univalence'' and 
``simplicial univalence'' in \cite[Theorem 3.3.7]{klvsimp}).
\end{remark}

\section{Completeness of simplicial objects}\label{seccompleteness}

In this section we discuss the strict and external notion of completeness for Segal objects in a model category
$\mathbb{M}$ as referred to in the Introduction. It is motivated by the conception of completeness as a 
simplicially enriched right lifting property of Segal objects, so that, first, their associated cotensors become right 
Quillen and hence yield well-behaved ``indexed quasi-categories'' over $\mathbb{M}$ (motivated by Joyal and Tierney's 
results), and, second, it may be used under additional conditions on the model category $\mathbb{M}$ as a locality 
condition in the construction of an associated model structure on $s\mathbb{M}=\mathbb{M}^{\Delta^{op}}$ (as in Rezk's 
and later Barwick's work).

Let $\mathbf{S}$ be the category of simplicial sets. First, given a simplicial object $X$ in a model category $\mathbb{M}$, consider 
it as a functor $X\colon\Delta\rightarrow\mathbb{M}^{op}$ together with left Kan extension
\[\xymatrix{
\Delta\ar[r]^X\ar[d]_{y} & \mathbb{M}^{op} \\
\mathbf{S}\ar@{-->}[ur]_(.6){\blank\setminus X:=\mathrm{Lan}_{y} X}.
}\]
The extension is given pointwise by the colimit
$A\setminus X\cong\mathrm{colim}({y/A}\rightarrow\Delta\xrightarrow{X}\mathbb{M}^{op})$. It hence computes the limit of 
the diagram $X\colon\Delta^{op}\rightarrow\mathbb{M}$ with weight $A$ in $\mathbb{M}$. The functor $\blank\setminus X$ 
comes with a right adjoint $X/\blank$, that is, its associated cotensor. Being the unique colimit preserving extension 
of $X$ along $y$, one can describe $\blank\setminus X$ furthermore as the end
\[A\setminus X:=\int_{[n]\in\Delta}X_n^{A_n}.\]
Hence, the functors $\blank\setminus X\colon\mathbf{S}\rightarrow\mathbb{M}^{op}$ and
$X/\blank\colon\mathbb{M}^{op}\rightarrow\mathbf{S}$ are induced by the parametrized right adjoints of the box product 
$\Box\colon\mathbf{S}\times\mathbb{M}\rightarrow s\mathbb{M}$ defined and explained in \cite[Section 7]{jtqcatvsss}.

We now define completeness in terms of this weighted limit functor $\blank\setminus X$. Therefore, we adapt the notation 
of \cite{jtqcatvsss} and say that a map $f\colon X\rightarrow Y$ of simplicial sets is, first, a \emph{mid-fibration} 
(or \emph{inner fibration}) if it has the right lifting property against all inner horn inclusions, and second, a 
\emph{quasi-fibration} (or \emph{categorical fibration}) if it is a fibration in the Joyal model structure
$(\mathbf{S},\mathrm{QCat})$.

The walking isomorphism
\[I[1]:=(\xymatrix{\bullet\ar@/^/[r]\ar@{}[r]|{\cong}  & \bullet\ar@/^/[l]})\in\mathbf{Gpd}.\]
induces an interval object $J\:=N(I[1])$ in the model category $(\mathbf{S},\mathrm{QCat})$ for 
quasi-categories. It has the property that a mid-fibration $p\colon \mathcal{C}\rightarrow \mathcal{D}$ between
quasi-categories is a fibration in $(\mathbf{S},\mathrm{QCat})$ if and only if it has the right lifting property against 
the endpoint inclusion $\{1\}\colon\Delta^0\rightarrow J$. This follows from \cite[Corollary 2.4.6.5]{luriehtt}. 
\begin{definition}\label{defcompleteJ}
A Reedy fibrant Segal object $X$ in $\mathbb{M}$ is \emph{complete} if the map
\[\{1\}\setminus X\colon J\setminus X\rightarrow X_0\]
is an acyclic fibration in $\mathbb{M}$.
\end{definition}

The definition is justified by the following proposition.

\begin{proposition}\label{lemmarQuillenchar}
Let $X$ be a simplicial object in $\mathbb{M}$ and consider its cotensor
\[X/\blank\colon\mathbb{M}^{op}\rightarrow\mathbf{S}.\]
\begin{enumerate}
\item The cotensor preserves acyclic fibrations if and only if $X$ is Reedy fibrant.
\item Suppose $X$ is Reedy fibrant. Then the cotensor takes fibrations in $\mathbb{M}^{op}$ to mid-fibrations in $\mathbf{S}$ if and 
only if $X$ is a Segal object.
\item Suppose $X$ is a Reedy fibrant Segal object. Then the cotensor takes fibrations in $\mathbb{M}^{op}$ to quasi-fibrations in
$\mathbf{S}$ if and only if $X$ is complete.
\end{enumerate}
Thus, a simplicial object $X$ in $\mathbb{M}$ is a complete Segal object if and only if its associated cotensor is a right Quillen 
functor of the form
\[X/\blank\colon\mathbb{M}^{op}\rightarrow(\mathbf{S},\mathrm{QCat}).\]
Its Quillen left adjoint is given by its associated weighted limit functor $\blank\setminus X$.
\end{proposition}
\begin{proof}
Let $X$ be a simplicial object in $\mathbb{M}$. We prove Part 1. If $\delta^n\colon\partial\Delta^n\rightarrow\Delta^n$ 
denotes the $n$-th boundary inclusion, then the map of associated weighted limits $\delta^n\setminus X$ in $\mathbb{M}$ 
is exactly the $n$-th matching map of $X$. Thus, by definition, $X$ is Reedy fibrant if and only if the contravariant 
functor $\blank\setminus X$ takes all boundary inclusions in $\mathbf{S}$ to fibrations in $\mathbb{M}$. Since
$\blank\setminus X$ preserves colimits, this in turn holds if and only if $\blank\setminus X$ takes 
all monomorphisms in $\mathbf{S}$ to fibrations in $\mathbb{M}$. By adjointness, this means that $X$ is Reedy fibrant if 
and only if the right adjoint $X/\blank\colon\mathbb{M}^{op}\rightarrow\mathbf{S}$ preserves acyclic fibrations.

For Part 2., assume $X$ is Reedy fibrant and let $\mathrm{sp}_n\colon\mathrm{Sp}_n\rightarrow\Delta^n$ be the $n$-th 
spine inclusion. That is, the union of the inclusions $j^i\colon \Delta^1\rightarrow \Delta^n$ for $0\leq i<n$ defined 
in Section~\ref{secunivsegal} (called the ``$n$-chain'' in \cite[Section 1]{jtqcatvsss}). Then the Segal maps
$\xi_n\colon X_n\rightarrow X_1\times_{X_0}\dots\times_{X_0} X_1$ are fibrations, and they are isomorphic to the maps
$\mathrm{sp}_n\setminus X$ of associated weighted limits. Thus, $X$ is a Segal object if and only if the contravariant 
functor $\blank\setminus X$ takes all spine inclusions in $\mathbf{S}$ to acyclic fibrations in $\mathbb{M}$. By 
\cite[Lemma 1.21, Lemma 3.5]{jtqcatvsss}, this in turn holds if and only if $\blank\setminus X$ takes all inner horn 
inclusions in $\mathbf{S}$ to acyclic fibrations in $\mathbb{M}$. Again by adjointness, this means that $X$ is a Segal 
object if and only if the right adjoint $X/\blank\colon\mathbb{M}^{op}\rightarrow\mathbf{S}$ takes fibrations to
mid-fibrations.

For Part 3., suppose $X$ is a Reedy fibrant Segal object. Suppose $X/\blank$ takes fibrations to quasi-fibrations. We want to show 
that the map $\{1\}\setminus X$ is an acyclic cofibration in $\mathbb{M}^{op}$. Thus, let $p$ be a fibration in
$\mathbb{M}^{op}$. Then $\{1\}\setminus X$ lifts against $p$ if and only if $\{1\}$ lifts against $X/p$ in $\mathbf{S}$. 
But the latter holds, because $X/p$ is a quasi-fibration by assumption.

Vice versa, suppose $X$ is complete. We want to show that
$X/\blank\colon\mathbb{M}^{op}\rightarrow(\mathbf{S},\mathrm{QCat})$ is a right Quillen functor. Therefore, it suffices 
to show that $X/\blank$ takes fibrations between fibrant objects to quasi-fibrations
(\cite[Proposition 7.15]{jtqcatvsss}). Thus, let
$p\colon E\twoheadrightarrow B$ be a fibration between fibrant objects in $\mathbb{M}^{op}$. Then by Part 2., the map
$X/p\colon X/E\rightarrow X/B$ is a mid-fibration between mid-fibrant simplicial sets (i.e.\ quasi-categories). Such 
mid-fibrations are quasi-fibrations if and only if they have the right lifting property against the endpoint inclusion
$\{1\}\colon\Delta^0\rightarrow J$. But the latter follows again by adjointness, since the map $\{1\}\setminus X$ lifts 
against $p$ in $\mathbb{M}^{op}$ by completeness of $X$.
\end{proof}

\begin{remark}
Recall that Rezk's model structure $(s\mathbf{S},\mathrm{S})$ for Segal spaces on the category of simplicial spaces is 
defined as the left Bousfield localization of the (vertical) Reedy model structure $(s\mathbf{S},\mathrm{R})$ at the
\emph{horizontally constant} maps $c_h(\mathrm{sp}_n)\colon c_h(\mathrm{Sp}_n)\rightarrow c_h(\Delta^n)$ induced by the 
spine inclusions in $\mathbf{S}$. The model structure $(s\mathbf{S},\mathrm{CS})$ for complete Segal spaces is a further 
localization at the single map $c_h(\{1\})\colon c_h(\Delta^0)\rightarrow c_h(J)$ analogously induced via the horizontal 
embedding $c_h\colon\mathbf{S}\rightarrow s\mathbf{S}$. All three of these model structures are simplicially enriched in 
such a way that for every simplicial set $A$ and every simplicial space $Y$ we have natural isomorphisms
\[s\mathbf{S}(c_h(A),Y)\cong A\setminus Y\]
of simplicial sets. This means that both the Segal conditions as given in Definition~\ref{defsegalob}.2 and the 
completeness condition of Segal objects as defined in Definition~\ref{defcompleteJ} reduce exactly to a simplicially 
enriched locality condition in the case $\mathbb{M}=(\mathbf{S},\mathrm{Kan})$.
\end{remark}

In the following, we will argue that we can replace the simplicial set $J$ with a combinatorially simpler version $K$. For a Segal 
object $X$, this will make the comparison of completeness to univalence as defined in Section~\ref{secunivsegal} merely a matter of 
comparing $\mathrm{Equiv}(X)$ (a fiber product of homotopy-equalizers) to the $K$-weighted limit $K\setminus\mathbb{R}X$ (a fiber 
product of corresponding strict equalizers) of the Reedy fibrant replacement $\mathbb{R}X$.
Therefore, we make note of the following lemma.

\begin{lemma}\label{lemmacompletenessaltdef}
Let $f\colon C\hookrightarrow D$ be a monomorphism in $\mathbf{S}$ such that a mid-fibration between
quasi-categories is a quasi-fibration if and only if it has the right lifting property against $f$. Then a Reedy fibrant 
Segal object $X$ in a model category $\mathbb{M}$ is complete in the sense of Definition~\ref{defcompleteJ} if and only 
if the fibration $f\setminus X$ is acyclic.
\end{lemma}
\begin{proof}
This follows directly from the proof of Proposition~\ref{lemmarQuillenchar}.3.
\end{proof}

This means that we can replace the endpoint inclusion $\Delta^0\rightarrow J$ in Definition~\ref{defcompleteJ} with any 
cofibration $f\colon C\hookrightarrow D$ that distinguishes mid-fibrations from quasi-fibrations between
quasi-categories. This proves useful since the combinatorial structure of $J$ is rather complicated. Rezk has observed 
in \cite[Section 11]{rezk} that $J$ is an $\omega$-sequential colimit $\bigcup_{n\geq 1}J^{(n)}$
where the finite approximations $J^{(n)}$ may be understood as the walking higher half adjoint equivalences 
with some degenerate side conditions.

Therefore, let $\Delta^1=(0\xrightarrow{f}1)$ and consider the simplicial set $K:=J^{(2)}\sqcup_{\Delta^1} J^{(2)}$ 
instead, where $J^{(2)}\subset J$ is given by the pushout
\begin{align}\label{diagdefJ2}
\begin{gathered}
\xymatrix{
\Lambda_0^2\ar@{^(->}[r]^{h^2_0}\ar[d]_{(f,s^00)}\ar@{}[dr]|(.7){\pos} & \Delta^2\ar[d]\\
\Delta^1\ar@{^(->}[r]_f & J^{(2)}.
}
\end{gathered}
\end{align}
It denotes the ``walking'' left invertible map
\[\xymatrix{
 & 1\ar@/^/[dr]^f & \\
0\ar@/^/[ur]^g\ar@/_/[rr]_{s^00} & & 0 
}\]
with exactly one non-degenerate 2-simplex as depicted above. And so $K\in\mathbf{S}$ is the ``walking'' 
biinvertible map depicted as follows.
\[\xymatrix{
 & 1\ar@/^/[dr]^f\ar@/^/[rr]^{s^01} & & 1\\
0\ar@/^/[ur]^g\ar@/_/[rr]_{s^00} & & 0\ar@/_/[ur]_h
}\]
For every quasi-category $\mathcal{C}$, the restriction of maps $K\rightarrow\mathcal{C}$ to either non-degenerate edge
in $K$ are exactly the equivalences in $\mathcal{C}$. Indeed, $K$ is exactly the interval $\Delta^1[f^{-1}]$ defined in 
\cite[Definition 3.3.3]{cisinskibook}.

\begin{lemma}\label{lemmaKJ}
A mid-fibration $p\colon \mathcal{C}\rightarrow\mathcal{D}$ between quasi-categories is a quasi-fibration if and only if 
it has the right lifting property against the endpoint inclusion $\Delta^0\rightarrow K$.
\end{lemma}
\begin{proof}
By \cite[Corollary 2.4.6.5]{luriehtt}, a mid-fibration $p\colon \mathcal{C}\rightarrow\mathcal{D}$ between
quasi-categories $\mathcal{C}$ and $\mathcal{D}$ is a quasi-fibration if and only if for every equivalence
$f\colon d\rightarrow d\sprime$ in $\mathcal{D}$ and every object $c\in\mathcal{C}$ with $p(c)=d$, there is an 
equivalence $\bar{f}\colon c\rightarrow c\sprime$ in $\mathcal{C}$ with $p(\bar{f})=f$. This holds if and only if
$p\colon \mathcal{C}\rightarrow\mathcal{D}$ has the right lifting property to the endpoint inclusion
$\Delta^0\rightarrow K$.
\end{proof}

\begin{corollary}\label{defcompletett}
Let $X$ be a Reedy fibrant Segal object in $\mathbb{M}$. Then $X$ is complete in the sense of
Definition~\ref{defcompleteJ} if and only if the fibration $K\setminus X\twoheadrightarrow X_0$ induced by the endpoint 
inclusion $\{1\}\colon\Delta^0\rightarrow K$ is acyclic.
\end{corollary}
\begin{proof}
Follows directly from Lemma~\ref{lemmacompletenessaltdef} and Lemma~\ref{lemmaKJ}.
\end{proof}
\section{Comparison of univalence and completeness}\label{secunivvscompl}

In Section~\ref{secunivsegal} we have given a definition of univalence of Segal objects $X$ in type theoretic fibration categories
$\mathbb{C}$ that generalizes the definition of univalence of fibrations in $\mathbb{C}$ via an internal nerve construction. In 
Section~\ref{seccompleteness} we have considered the cotensors associated to Reedy fibrant Segal objects in a model category
$\mathbb{M}$, and have argued that their right Quillen condition yields a definition of completeness in a general model categorical 
context.

An intersection of these two frameworks is given by logical model categories in the sense of \cite{arndtkapulkin}. That 
is a model category $\mathbb{M}$ such that
\begin{enumerate}
\item $\mathbb{M}$ has the \emph{Frobenius property}, i.e.\ acyclic cofibrations are stable under pullback 
along fibrations, and
\item pullback $p^{\ast}$ along any fibration $p$ has a right adjoint $\prod_p$.
\end{enumerate}
Its associated category $\mathbb{C}:=\mathbb{M}^f$ of fibrant objects is type theoretic as defined in
Section~\ref{secunivsegal}. We assume furthermore that fibrant objects in $\mathbb{M}$ are cofibrant, in order to 
assure that the classes of weak equivalences and homotopy-equivalences in $\mathbb{C}$ coincide
(\cite[Lemma 1.3.4]{thesis}).\footnote{This additional assumption is also made in many examples of type theoretic model 
categories in \cite{shulmaninv}, and it is arguable that cofibrancy of fibrant objects should be part of the definition 
of any model category that is labelled ``logical'' or ``type theoretic''.} 

In such a logical model category $\mathbb{M}$, both conditions -- univalence and completeness of Segal objects $X$ -- 
compare an object of internal equivalences associated to $X$ to the path-object $P(X_0)$. On the one hand, univalence 
considers the type of fiberwise equivalences $\mathrm{Equiv}(X)$ given in Diagram~(\ref{defequtt}) which 
directly generalizes the type of fibrewise equivalences associated to a fibration (Lemma~\ref{lemmasimpunivalence}, 
Proposition~\ref{defunivalenceso}.2). On the other hand, completeness considers the weighted limit $K\setminus X$, where 
$K\in\textbf{S}$ denotes the truncated walking isomorphism as defined in Section~\ref{seccompleteness} 
(Corollary~\ref{defcompletett}).

In this section we show that the type $\mathrm{Equiv}(X)$ associated to a (sufficiently fibrant) Segal object is 
homotopy-equivalent to the $K$-weighted limit of its Reedy fibrant replacement $\mathbb{R}X$. In particular, 
univalence and completeness coincide for Reedy fibrant objects.

In the following $\mathbb{M}$ is a fixed logical model category in which all fibrant objects are cofibrant. Let
$\mathbb{C}:=\mathbb{M}^f$ be its category of fibrant objects. We choose not to work more generally with an arbitrary type 
theoretic fibration category $\mathbb{C}$ directly here (as $K$-weighted limits of fibrations do exist in such $\mathbb{C}$ and hence 
may be computed and compared in this context already) is the proof of Lemma~\ref{intc=simpc_interm} which assumes the existence of a 
Reedy fibrant replacement functor in the category $s\mathbb{M}$. It hence assumes the existence of finite colimits in $\mathbb{M}$ 
(see Remark~\ref{remreedycolims} for more details).

Let $X\in s\mathbb{C}$ be a Segal object and consider the limit
\[K\setminus X\cong \left(J^{(2)}\cup_{\Delta^1} J^{(2)}\right)\setminus X\cong(J^{(2)}\setminus X )\times_{X_1}(J^{(2)}\setminus X).\]
By definition we have $\Delta^n\setminus X\cong X_n$ for all $n\geq 0$. For the component $J^{(2)}\setminus X$
note that the pushout (\ref{diagdefJ2}) can be rewritten to compute
\[\begin{gathered}
\xymatrix{
\Lambda_0^2\ar[rr]^{(d^1d^1s^0,d^1)}\ar[d]_{\nabla}\ar@{}[drr]|(.7){\pos} & & \Delta^2\ar[d]\\
\Delta^1\ar@{^(->}[rr]_f & & J^{(2)}.
}
\end{gathered}
\]
It follows that $J^{(2)}\setminus X\hookrightarrow X_2$ is the equalizer of the parallel pair
\begin{equation}\label{typeoflinv}
\begin{gathered}
\xymatrix@R=3mm@C=8mm{
d_1\colon X_2\ar[rrr] & & &  X_1,\\
s\colon X_2\ar[r]^{d_1} & X_1\ar[r]^{d_1} & X_0\ar[r]^{s_0} & X_1.
}
\end{gathered}
\end{equation}
Since the composite boundary $J^{(2)}\setminus X\subseteq X_2\xrightarrow{d_1}X_1\xrightarrow{(d_1,d_0)}X_0\times X_0$ factors through the diagonal $\Delta\colon X_0\rightarrow X_0\times X_0$, the pullback
$\Delta^{\ast}J^{(2)}\setminus X$ is isomorphic to $J^{(2)}\setminus X$ itself. 

It follows that $J^{(2)}\setminus X$ is exactly the equalizer of the pair $(\Delta^{\ast}d_1, \Delta^{\ast}s)$ on the 
left hand side of Diagram~(\ref{diageqdj}). Data in $J^{(2)}\setminus X$ consists of two concatenable 
edges in $X_1$ together with a horizontal composition in $X_2$ of theirs which evaluates (by computing its 1-boundary) at the 
corresponding identity in $X$. Data in the object $\mathrm{Inv}X$ defined in Diagram~(\ref{diageqdj}) consists of two concatenable 
edges in $X$, together with a horizontal composition \emph{and a homotopy} between its evaluation at the 1-boundary and the 
corresponding identity in the space $X_1$. In the following, we will see that this discrepancy is remedied by Reedy fibrant 
replacement, so that the strict weighted limit $K\setminus \mathbb{R}X$ of $K$-shaped diagrams in the Reedy fibrant replacement of a 
Segal object $X$ represents the type $\mathrm{Equiv}(X)$ of internal equivalences in $X$. 

\begin{remark}
Let us assume for the moment being that the model category $\mathbb{M}$ is simplicially enriched. Recall that the 
category $s\mathbf{S}$ of simplicial spaces is simplicially enriched as well via pushforward of its cartesian closure 
along the evaluation at $[0]\in\Delta^{op}$.
Then one computes that the restriction \[\{\blank\bBox\Delta^0,X\}_{\mathbf{S}}\colon\mathbf{S}^{op}\rightarrow\mathbb{M}\]
of the simplicially enriched weighted limit functor $\{\blank,X\}_{\mathbf{S}}\colon (s\mathbf{S})^{op}\rightarrow\mathbb{M}$ along the horizontal inclusion $\blank\bBox\Delta^0\colon\mathbf{S}\rightarrow s\mathbf{S}$ is isomorphic to 
the ordinary Set-enriched weighted limit $\blank\setminus X\colon\mathbf{S}^{op}\rightarrow\mathbb{M}$.
In analogy to \cite[Theorem 4.4]{gambinosmplwghts}, using \cite[Example A.2.9.28]{luriehtt} or directly \cite[Proposition 7.33]{jtqcatvsss}, one can show that for every 
simplicial set $A$ this restricted simplicially enriched weighted limit functor
\[\{A\bBox\Delta^0,\blank\}_{\mathbf{S}}\colon s\mathbb{M}\rightarrow\mathbb{M}\]
is a right Quillen functor when $s\mathbb{M}$ is equipped with the Reedy model structure. In other words, for every 
simplicial set $A$ the $A$-weighted limit functor $A\setminus\blank\colon s\mathbb{M}\rightarrow\mathbb{M}$ is a right 
Quillen functor. It follows that the $A$-weighted homotopy limit of $X$ can be computed as the $A$-weighted limit of the 
Reedy fibrant replacement $\mathbb{R}X$ of $X$.

In Theorem~\ref{intc=simpc}, we will show that in all logical model categories the type $\text{Equiv}(X)$ of internal 
equivalences in a Segal object $X$ is equivalent to the $K$-weighted limit of the Reedy fibrant replacement of $X$.
In this sense, the type $\mathrm{Equiv}(X)$ associated to a Segal object $X$ in a logical model category $\mathbb{M}$ 
represents the $K$-weighted homotopy limit of $X$ in $\mathbb{M}$ whenever $\mathbb{M}$ is furthermore simplicial.
\end{remark}

Therefore, consider the following construction.

\begin{lemma}\label{intc=simpc_interm}
For every Segal object $X\in s\mathbb{C}$ there is a pointwise homotopy-equivalent Reedy fibrant Segal object $\tilde{X}$ in
$\mathbb{C}$ such that $X$ and $\tilde{X}$ coincide in degrees $\leq 1$, and such that there is a homotopy-equivalence
$\mathrm{Inv}X\simeq J^{(2)}/\tilde{X}$ defined over the pullback
\[\xymatrix{
X_1\times_{X_0^2} X_1\ar@{->>}[r]\ar@{}[dr]|(.3){\pbs}\ar@{->>}[d] & X_1\ar@{->>}[d]^{(d_1,d_0)} \\
X_1\ar@{->>}[r]_(.4){(d_0,d_1)} & X_0\times X_0
}\]
of conversely directed edges in $X$.
\end{lemma}

\begin{proof}
We define $\tilde{X}$ following the general recursive construction of Reedy fibrant replacement of simplicial objects. 
We start with $(\tilde{X})_0=X_0$ and $(\tilde{X})_1=X_1$ with the same boundaries and the same degeneracy, since $X$ is 
sufficiently fibrant. At the next degree, instead of taking an arbitrary factorization of the matching map
$\partial:=(d_2,d_1,d_0)\colon X_2\rightarrow M_2 X$ with $M_2 X=M_2(\tilde{X})$ as usual, we factor the pair
\[(\partial,s_0d_1d_1)\colon X_2\rightarrow M_2 X\times_{X_0}\Delta^{\ast}X_1\]
into an acyclic cofibration followed by a fibration in $\mathbb{M}$.
\begin{align}\label{defreedy2fac}
\begin{gathered}
\xymatrix{
 & \tilde{X}_2\ar@{->>}[dr]^{} & & \\
X_2\ar@{^(->}[ur]^{\rotatebox[origin=c]{45}{$\sim$}}\ar[rr]_(.35){(\partial,s_0d_1d_1)} & & M_2 X\times_{X_0}\Delta^{\ast}X_1\ar@{->>}[r]_(.65){\pi_1} & M_2 X.\\
}
\end{gathered}
\end{align}
Since $\Delta^{\ast}X_1$ is fibrant over $X_0$, the object $\tilde{X}_2$ also factors the boundary map
$\partial\colon X_2\rightarrow M_2 X$ into the same acyclic cofibration followed by the composite fibration
$\tilde{X}_2\twoheadrightarrow M_2 X\times_{X_0}\Delta^{\ast}X_1\twoheadrightarrow M_2 X$. We continue 
choosing the $(\tilde{X})_n$ for $n\geq 3$ inductively by standard procedure of Reedy fibrant replacement with no 
further restrictions (see e.g.\ \cite[Theorem 5.2.5]{hovey}). Thus, $\tilde{X}$ is a Reedy fibrant replacement of $X$.

By Remark~\ref{remsegalhmtpy}, the simplicial object $\tilde{X}$ is still a Segal object in $\mathbb{C}$. 
Since $X_2$ and $\tilde{X}_2$ are both fibrant over $X_0\times X_0$, their pullbacks along the diagonal
$\Delta\colon X_0\rightarrow X_0\times X_0$ are still homotopy-equivalent to each other.
If we denote the maps $d_1,s\colon X_2\rightarrow X_1$ from (\ref{typeoflinv}) defined for $\tilde{X}$ by
$\tilde{d_1},\tilde{s}\colon\tilde{X}_2\rightarrow\tilde{X}_1$, respectively, we obtain factorizations of (\ref{diageqdj}) as follows 
by construction.
\[\xymatrix{
\mathrm{Equalizer}(\Delta^{\ast}d_1,\Delta^{\ast}s)\ar[d]\ar[r]\ar@{}[dr]|(.3){\pbs} & \mathrm{Inv} X\ar@{}[dr]|(.3){\pbs}\ar[d]^{\rotatebox[origin=c]{90}{$\sim$}}\ar@{->>}[r] & \Delta^{\ast}X_2\ar[d]^{\rotatebox[origin=c]{90}{$\sim$}}\ar@/^3pc/[dd]^{(\Delta^{\ast}d_1,\Delta^{\ast}s)}\\
J^{(2)}/\tilde{X}\ar@{}[dr]|(.3){\pbs}\ar@{->>}[d]\ar@{^(->}[r]^{\sim} & \mathrm{Inv}(\tilde{X})\ar@{}[dr]|(.3){\pbs}\ar@{->>}[d]\ar@{->>}[r] & \Delta^{\ast}\tilde{X}_2\ar@{->>}[d]|{(\Delta^{\ast}\tilde{d_1},\Delta^{\ast}\tilde{s})} \\
\Delta^{\ast}X_1\ar@{^(->}[r]^{\sim}\ar@/_2pc/[rr]_{\mathrm{diag}}& P_{X_0}(\Delta^{\ast}X_1)\ar@{->>}[r] & \Delta^{\ast}X_1\times_{X_0}\Delta^{\ast}X_1
}\]
Hence, we see that whether to construct $\mathrm{Inv} X$ or $J^{(2)}\setminus\tilde{X}$ is a matter of 
fibrantly replacing either of the two legs of the outer square in (\ref{diageqdj}). We thus obtain a composite homotopy-equivalence of 
the form
\[\mathrm{Inv} X\simeq J^{(2)}\setminus\tilde{X}.\]
This equivalence can be seen to be defined over
$X_1\times_{X_0^2} X_1=\tilde{X}_1\times_{\tilde{X}_0^2}\tilde{X}_1$ by factoring the 
vertical map $(\Delta^{\ast}\tilde{d_1},\Delta^{\ast}\tilde{s})$ through the maps (\ref{defreedy2fac}) after reindexing 
along the diagonal $\Delta$ as well.
\end{proof}

\begin{remark}\label{remreedycolims} In this remark we address the technical insufficiency of general type theoretic fibration 
categories for the proof of Lemma~\ref{intc=simpc_interm}, and hence for this entire section.
Namely, the concise reference to ``the standard procedure of Reedy fibrant replacement" (given in \cite[Theorem 5.2.5]{hovey}) in 
order to define the simplicial object $\tilde{X}$ relies on the existence of certain 
colimits. More precisely, it requires the construction of latching objects and 
pushouts along latching maps of the form $L_n X\rightarrow X_n$ for sufficiently fibrant $X\in s\mathbb{C}$, $n\geq 0$, 
in the ambient model category $\mathbb{M}$. These finite colimits do not so much assure Reedy fibrancy of
$\tilde{X}$ (which is imposed by fibrancy of the dual matching maps and which do exist in any type theoretic fibration 
category $\mathbb{C}$), but rather that $\tilde{X}$ is a (strict) simplicial object in $\mathbb{C}$ in the first place. 
Indeed, a recursive definition of $\tilde{X}$ via factorization of only the composite
$X_n\rightarrow M_n X\rightarrow M_n\tilde{X}$ into some
$X_n\overset{\sim}{\hookrightarrow}\tilde{X}_n\twoheadrightarrow M_n\tilde{X}$ does yield boundaries and degeneracies 
for $\tilde{X}$ up to level $n$ that satisfy all simplicial identities on the nose but potentially the ones which 
govern the composition of higher degeneracies. These can only be verified to hold up to homotopy, and there is no 
apparent way to strictify. This is the reason for the resort to semi-simplicial objects and 
the study of semi-Segal objects in type theory, see e.g.\ \cite[Section 4.6]{capriottithesis}. While such semi-simplicial 
objects allow for the development of a notion of univalence as in Section~\ref{secunivsegal} (see
\cite[Section 4.6.4]{capriottithesis}, there called completeness), the definition of completeness via
Corollary~\ref{defcompletett} intrinsically uses the degeneracies $s_0\colon [0]\rightarrow[1]$ and
$s_i\colon [1]\rightarrow [2]$. Definition~\ref{defcompleteJ} itself requires all degeneracies even.

A compromise work-around with semi-simplicial objects which exhibit level $0$ and $1$ degeneracies appears to be 
possible in principle, but makes the relation to the constructions in Section~\ref{seccompleteness} rather awkward, on top of being a 
rather awkward notion on its own. Indeed, Lemma~\ref{intc=simpc_interm} has to be proven 
for simplicial objects in logical model categories explicitly in either case in order to make the connection to
Section~\ref{seccompleteness}, and so it appears that little would be gained	 by the extra generality. 
\end{remark}

It is reasonable to expect that the constructions used to define univalence and completeness are
homotopy-invariant, and they are indeed in the following sense.
\begin{lemma}\label{compltehtyinv}
Suppose $X$, $Y\in s\mathbb{C}$ are sufficiently fibrant and pointwise homotopy-equivalent.
\begin{enumerate}
\item There is a homotopy-equivalence $\mathrm{Inv} X\simeq\mathrm{Inv} Y$ such that the square
\[\xymatrix{
\mathrm{Inv}X\ar@{->>}[d]_{(\mathrm{Linv}(X),\mathrm{Rinv}(X))}\ar[r]^{\simeq} & \mathrm{Inv}Y\ar@{->>}[d]^{{(\mathrm{Linv}(Y),\mathrm{Rinv}(Y))}} \\
X_1\times_{X_0^2} X_1\ar[r]_{\simeq} & Y_1\times_{Y_0^2} Y_1
}\]
commutes.
\item If $X$ and $Y$ are furthermore Reedy fibrant, then $J^{(2)}\setminus X\simeq J^{(2)}\setminus Y$ analogously.
\end{enumerate}
\end{lemma}
\begin{proof}
Part 1.\ is easily seen from Diagram (\ref{diageqdj}) by exploiting sufficient fibrancy of $X$ and $Y$ and right-properness of
$\mathbb{M}$.

For part 2., given an equivalence $X\xrightarrow{\sim} Y$, again by virtue of right-properness of $\mathbb{M}$ we obtain an 
equivalence of pullbacks as follows.
\[\xymatrix{
		& J^{(2)}\setminus Y\ar[rr]\ar@{->>}'[d][dd]\ar@{}[ddrr]|(.2){\pbs} &		& Y_2\ar@{->>}[dd] \\
J^{(2)}\setminus X\ar[rr]\ar@{->>}[dd]\ar[ur]^{\rotatebox[origin=c]{35}{$\sim$}}\ar@{}[ddrr]|(.2){\pbs} &		  & X_2\ar@{->>}[dd]\ar[ur]^{\rotatebox[origin=c]{35}{$\sim$}} &	  \\
		& Y_1\ar'[r][rr]	  &		& \Lambda_0^2\setminus Y\\
X_1\ar[rr]\ar[ur]^{\rotatebox[origin=c]{35}{$\sim$}}	  & 		  & \Lambda_0^2\setminus X\ar[ur]^{\rotatebox[origin=c]{35}{$\sim$}}
}\]
\end{proof}

The assumption of Reedy fibrancy in the proof of Lemma~\ref{compltehtyinv} cannot be replaced with sufficient fibrancy 
only, because it uses that both functors
$\blank\setminus X, \blank\setminus Y\colon\mathbf{S}^{op}\rightarrow\mathbb{C}$ send the 2-horn inclusion
$\Lambda_0^2\hookrightarrow\Delta^2$ to a fibration.

\begin{proposition}\label{comparelevel2}
Let $X\in s\mathbb{C}$ be a Segal object in $\mathbb{C}$. Then for any Reedy fibrant replacement
$r\colon X\rightarrow \mathbb{R}X$, we obtain a commutative square as follows.
\[
\xymatrix{
\mathrm{Inv}(X)\ar[r]^{\simeq}\ar[d]_{(\mathrm{Linv(X)},\mathrm{Rinv}(X))} & J^{(2)}\setminus \mathbb{R}X\ar[d]^{(d_2,d_0)} \\
X_1\times_{X_0^2}X_1\ar[r] & \mathbb{R}X_1\times_{\mathbb{R}X_0^2}\mathbb{R}X_1
}
\]
\end{proposition}

\begin{proof}
Let $X$ be a Segal object and let $\mathbb{R}X$ be a Reedy fibrant replacement of $X$. 
By Lemma \ref{compltehtyinv}.1, we obtain a homotopy-equivalence
\[\mathrm{Inv}X\simeq\mathrm{Inv}\mathbb{R}X\]
over the two instances of $\mathrm{Linv}$ and $\mathrm{Rinv}$.
Then consider $\widetilde{(\mathbb{R}X)}$, the object from Lemma \ref{intc=simpc_interm} constructed from
$\mathbb{R}X$ (instead from $X$) such that
$\mathrm{Linv}\mathbb{R}X\simeq J^{(2)}\setminus\widetilde{(\mathbb{R}X)}$ over
$\mathbb{R}X_1\times_{\mathbb{R}X_0^2} \mathbb{R}X_1$. But both $\mathbb{R}X$ and
$\widetilde{(\mathbb{R}X)}$ are Reedy fibrant, hence an application of Lemma~\ref{compltehtyinv}.2 finishes the proof.
\end{proof}

We conclude this section with the main result.

\begin{theorem}\label{intc=simpc}
Let $X\in s\mathbb{C}$ be a Segal object. Then the following are equivalent. 
\begin{enumerate}
\item $X$ is univalent.
\item For any Reedy fibrant replacement $\mathbb{R}X$ of $X$, $\mathbb{R}X$ is complete.
\end{enumerate}
In particular, if $X$ is Reedy fibrant, $X$ is univalent if and only if it is complete.
\end{theorem}

\begin{proof}
Given a Segal object $X\in s\mathbb{C}$ together with a Reedy fibrant replacement
$X\overset{\sim}{\hookrightarrow}\mathbb{R}X$ in $s\mathbb{C}$, by Lemma~\ref{compltehtyinv} and 
Proposition~\ref{comparelevel2} we obtain a diagram of the following form.
\[\xymatrix{
 & \mathrm{Inv}X\ar@{->>}[dd]|\hole_(.3){\mathrm{RInv}(X)}\ar[rr]^{\simeq} & & \mathrm{Inv}\mathbb{R}X\ar@{->>}[dd]|\hole_(.3){\mathrm{RInv}(\mathbb{R}X)}\ar[rr]^{\simeq} & & J^{(2)}\setminus\mathbb{R}X\ar@{->>}@/^2pc/[ddll]^{d_0\setminus\mathbb{R}X} \\
\mathrm{Inv}X\ar@{->>}@/_/[dr]_{\mathrm{LInv}(X)}\ar[rr]^(.6){\simeq} & & \mathrm{Inv}\mathbb{R}X\ar@{->>}@/_/[dr]_(.3){\mathrm{RInv}(\mathbb{R}X)}\ar[rr]^(.6){\simeq} & & J^{(2)}\setminus\mathbb{R}X\ar@{->>}@/^/[dl]_{d_2\setminus\mathbb{R}X} & \\
 & X_1\ar@{^(->}[rr]^{\sim} & & \mathbb{R}X_1
}\]
This induces a homotopy-equivalence on pullbacks
\[\mathrm{Equiv}X\xrightarrow{\simeq}\mathrm{Equiv}\mathbb{R}X\xrightarrow{\simeq}(J^{(2)}\setminus\mathbb{R}X\times_{\mathbb{R}X_1} J^{(2)}\setminus\mathbb{R}X) \cong K\setminus\mathbb{R}X\]
over $X_1\overset{\sim}{\hookrightarrow}\mathbb{R}X_1$. Therefore
\[\xymatrix{
\mathrm{Equiv}X\ar[r]^{\simeq}\ar@{->>}[d] & \mathrm{Equiv}\mathbb{R}X\ar[r]^{\simeq}\ar@{->>}[d] & K\setminus\mathbb{R}X\ar@{->>}@/^/[dl] \\
X_0\ar@{^(->}[r]^{\sim} &\mathbb{R}X_0 
}\]
commutes, too, and the statement follows directly from 2-out-of-3 for homotopy-equivalences.
\end{proof}

\begin{corollary}\label{intc=simpcforfibs}
Let $p\colon E\twoheadrightarrow B$ be a fibration in $\mathbb{C}$. Then the following are equivalent.
\begin{enumerate}
\item The fibration $p$ is a univalent type family in $\mathbb{C}$.
\item The nerve $N(p)$ is a univalent Segal object.
\item For any Reedy fibrant replacement $X_p$ of $N(p)$, the Segal object $X_p$ is complete.
\end{enumerate}
\end{corollary}

\begin{proof}
Follows immediately from Lemma~\ref{lemmasimpunivalence} and Theorem~\ref{intc=simpc}.
\end{proof}

\begin{remark}\label{remrpmodcats}
If one \emph{defines} univalence of a Segal object in terms of the object $\mathrm{Equiv}(X)$ and homotopy-invertibility of its 
associated lift (\ref{diagequivdeg}), then the proof of Theorem~\ref{intc=simpc} applies to every model category that is right-proper 
(which is needed for Lemma~\ref{compltehtyinv}). Although in this case univalence lacks syntactical motivation, it still corresponds 
to the $(\infty,1)$-categorical notion of completeness as studied by Rasekh (\cite[Definition 2.47]{rasekh}).
Theorem~\ref{intc=simpc} is therefore still meaningful in that it shows equivalence of $(\infty,1)$-categorical completeness 
of a Reedy fibrant Segal object $X\in s\mathbb{M}$, and completeness of $X$ in the sense of
Definition~\ref{defcompleteJ}.
\end{remark}

\section{Associated completions and localizations}\label{secunivcompl}

In \cite{univalentcompletion}, van den Berg and Moerdijk introduced a notion of \emph{univalent completion} of a 
fibration $p$ to a univalent fibration $u(p)$ in the Quillen model structure $(\mathbf{S},\mathrm{Kan})$ on simplicial 
sets. Likewise, there is a notion of Rezk-completion of Segal spaces, arising as the fibrant replacement of a Segal 
space in the model category $(s\mathbf{S},\mathrm{CS})$ for complete Segal spaces. In this section, we generalize both 
notions and use Theorem~\ref{intc=simpc} to show that univalent completion is a special case of Rezk-completion whenever 
both are defined.

We therefore introduce a notion of BM-equivalences between fibrations in suitable type theoretic fibration categories
$\mathbb{C}$ inspired by \cite{univalentcompletion}, and a notion of DK-equivalences between Segal objects which 
naturally generalizes the notion of Dwyer-Kan equivalences in \cite{rezk}.
This is the content of Sections~\ref{secsubunivcompl} and \ref{secsubrezkcompl}, respectively. We then show in
Section~\ref{secsubcomparecompls} that for such suitable type theoretic fibration categories $\mathbb{C}$, the nerve 
functor assignment $p\mapsto N(p)$ from Section~\ref{secunivsegal} gives rise to a functor between relative categories 
with respect to these two notions of equivalences, such that $p\rightarrow u(p)$ is the univalent completion of a 
fibration $p$ if and only if $N(p)\rightarrow N(u(p))$ is the Rezk-completion of $N(p)$. Motivated by this 
correspondence, and Rezk's and Barwick's characterization of complete Segal objects as the homotopically
DK-local Reedy fibrant Segal objects in a large class of model categories, we show in Section~\ref{secsubBMloc} 
that whenever $\mathbb{C}$ arises from a logical model category, then the univalent fibrations are exactly the BM-local fibrations in 
a suitable homotopical sense whenever univalent completion exists.

The main results of this section are summarised in the following theorem.

\begin{theorem}\label{thm5sum}
Let $\mathbb{C}$ be a type theoretic fibration category with $(-1)$-truncations. Then the nerve construction induces a pullback square
\begin{align*}
\begin{gathered}
\xymatrix{
(\mathcal{UF}^{\times},\mathcal{W})\ar[r]^(.45)N\ar@{^(->}[d]\ar@{}[dr]|(.3){\pbs} & (\mathrm{US}(\mathbb{C})^{\times},\mathcal{W})\ar@{^(->}[d]\\
(\mathcal{F}^{\times},\mathrm{BM})\ar[r]_(.45)N & (\mathrm{S}(\mathbb{C})^{\times},\mathrm{DK})
}
\end{gathered}
\end{align*}
of relative categories such that
\begin{enumerate}
\item the vertical arrows are relative inclusions;
\item the relative functor $N$ reflects DK-equivalences with univalent codomain to BM-equivalences with univalent 
codomain (and hence reflects Rezk-completion to univalent completion).
\item Suppose $\mathbb{C}$ furthermore satisfies function extensionality and fibrations in $\mathbb{C}$ have the homotopy 
lifting property. Given a univalent fibration $\pi\in\mathbb{C}$, the $\pi$-small univalent fibrations in $\mathbb{C}$ are 
exactly the $\pi$-small fibrations which are internally local with respect to all $\pi$-small $BM$-equivalences.
\end{enumerate}
\end{theorem}
\begin{proof}
Parts 1 and 2 are Proposition~\ref{thm5sum1}, Part 3 is Proposition~\ref{propintBMloc}.
\end{proof}

\begin{corollary}\label{cor5sum}
Suppose $\mathbb{C}$ furthermore arises as the category of fibrant objects of a logical model category whose fibrant 
objects are cofibrant. Then the Reedy fibrant nerve construction from Corollary~\ref{intc=simpcforfibs} 
induces a pullback square
\begin{align*}
\begin{gathered}
\xymatrix{
(\mathcal{UF}^{\times},\mathcal{W})\ar[r]^(.45){\mathbb{R}N}\ar@{^(->}[d]\ar@{}[dr]|(.3){\pbs} & (\mathrm{CS}(\mathbb{C})^{\times},\mathcal{W})\ar@{^(->}[d]\\
(\mathcal{F}^{\times},\mathrm{BM})\ar[r]_(.45){\mathbb{R}N} & (\mathrm{S}(\mathbb{C})_f^{\times},\mathrm{DK})
}
\end{gathered}
\end{align*}
of relative categories such that
\begin{enumerate}
\item the vertical arrows are relative inclusions;
\item the relative functor $\mathbb{R}N$ reflects DK-equivalences with complete codomain to BM-equivalences with univalent 
codomain (and hence reflects Rezk-completion to univalent completion);
\item given a univalent fibration $\pi\in\mathbb{C}$, the $\pi$-small univalent fibrations in $\mathbb{C}$ are 
exactly the $\pi$-small fibrations which are internally local with respect to all $\pi$-small $BM$-equivalences.
\end{enumerate}
\end{corollary}
\begin{proof}
Parts 1 and 2 are Proposition~\ref{thmcmpcompletion}. Part 3 follows from Theorem~\ref{thm5sum}.3, as both function 
extensionality (\cite[Remark 5.10]{shulmaninv}) and the homotopy lifting property are automatically satisfied in the given special 
case.
\end{proof}

Thus, Corollary~\ref{cor5sum}.3 relates directly to the fact 
that in a large class of model categories the complete Segal objects are exactly the homotopically DK-local Reedy 
fibrant Segal objects (see Remark~\ref{remend} for a further discussion).

\subsection{$(-1)$-truncations}

As before, in this section we work with type theoretic fibration categories $\mathbb{C}$, but we will additionally assume that
$\mathbb{C}$ has $(-1)$-truncations (Definition~\ref{defpbstablehtyimagefac}). Therefore we make the following definitions and 
observations. It may be worth to note that all statements in this subsection apply to Joyal's tribes in general, but we will stick to 
type theoretic fibration categories as these are our case of interest and allow for a more direct use of references. Thus, let
$\mathbb{C}$ be a type theoretic fibration category.

A fibration $p\colon E\twoheadrightarrow B$ in $\mathbb{C}$ is $(-1)$-truncated if and only if the path-object fibration 
$P_BE\twoheadrightarrow E\times_B E$ has a section. That is, if and only if the fibration
$P_BE\twoheadrightarrow E\times_B E$ is an acyclic fibration, or equivalently if and only if the diagonal
$\Delta\colon E\rightarrow E\times_B E$ is a homotopy-equivalence. Say that a map $f\colon E\rightarrow B$ is $(-1)$-truncated if, for 
any factorization $f=j p$ into an acyclic cofibration $j$ and a fibration $p$, the fibration $p$ is $(-1)$-truncated.

\begin{definition}\label{defhtyimage}
A \emph{$(-1)$-truncation} of a map $f\colon A\rightarrow B$ in a type theoretic fibration category $\mathcal{C}$ is a factorization $f=pj$ such that
\begin{enumerate}
\item $p\colon E\twoheadrightarrow B$ is a $(-1)$-truncated fibration in $\mathbb{C}$, and
\item whenever $f$ exhibits another factorization $f=qi$ for a $(-1)$-truncated morphism $q\colon C\rightarrow B$ in
$\mathbb{C}$, there is a morphism $l\colon E\rightarrow C$ together with a homotopy between $p\colon E\rightarrow B$ and 
$q\circ l\colon E\rightarrow B$.
\end{enumerate}
\end{definition}

In virtue of Lemma~\ref{lemmamono} below, data given as in Definition~\ref{defhtyimage}.2 implies homotopy-commutativity 
of not only the right hand side but both triangles in the square
\begin{align}\label{diagdefhtyimage}
\begin{gathered}
\xymatrix{
  & E\ar@{->>}[dr]^p\ar[dd]_l & \\
A\ar[dr]_{i}\ar[ur]^j &  & B. \\
  & E\sprime\ar[ur]_{q} & 
}
\end{gathered}
\end{align}
Furthermore, the condition on the right map in a $(-1)$-truncation to be a fibration is technically redundant, since any
$(-1)$-truncated morphism may be further factored through an acyclic cofibration followed by a $(-1)$-truncated 
fibration, and acyclic cofibrations admit homotopy-inverses (\cite[Lemma 3.6]{shulmaninv}). It yet saves us to perform this reduction 
every time it is necessary.

\begin{lemma}\label{lemmamono}
Let $\mathbb{C}$ be a type theoretic fibration category. A map $f\colon E\rightarrow B$ in $\mathbb{C}$ is $(-1)$-truncated if and 
only if the natural map
\[\mathtt{tr}_f\colon PE\rightarrow PB\times_{(B\times B)}(E\times E)\]
is a homotopy-equivalence. In other words, $f$ is $(-1)$-truncated if and only if transport of paths along $f$ yields an 
equivalence between paths in $E$ and paths in $B$ between endpoints in $E$. 
\end{lemma}

\begin{proof}
Since a factorization $f=pj$ into an acyclic cofibration $j$ and a fibration $p$ gives a factorization of
$\mathtt{tr}_f\colon PE\rightarrow PB$ into a homotopy-equivalence and $\mathtt{tr}_p$, without loss of generality the 
map $f$ is a fibration. The ``only if'' direction is shown in \cite[Proposition 1.4.6]{thesis}. For the other direction, 
assume $\mathtt{tr}_f$ is an equivalence. The two projections of the composition
\begin{align}\label{equlemmamono}
E\times_B E\rightarrow E\times E\rightarrow B\times B
\end{align}
yield the same map into $B$. That means the composition (\ref{equlemmamono}) factors through the diagonal of $B$ and 
hence yields a lift to $PB$ represented via the constant path $r_B\colon B\rightarrow PB$. In particular, 
the natural map $\iota_B\colon E\times_B E\rightarrow E\times E$ lifts to the pullback
$PB\times_{(B\times B)}(E\times E)$. Since $\mathtt{tr}_f$ is an equivalence, we obtain a homotopy
$H\colon\pi_1\iota_B\sim\pi_2\iota_B$ in $E$ such that $\mathtt{tr}_fH=r_Bf\pi_1$. Hence, by
\cite[Lemma 3.9]{shulmaninv}, we obtain a homotopy $H\sprime\colon\pi_1\iota_B\sim\pi_2\iota_B$ in $(\mathbb{C}/B)^f$, 
or in other words a section to $P_BE\twoheadrightarrow E\times_BE$. Thus, the fibration $f\colon E\twoheadrightarrow B$ 
is $(-1)$-truncated.
\end{proof}

\begin{corollary}\label{corhtyimage1}
Let $\mathbb{C}$ be a type theoretic fibration category. 
\begin{enumerate}
\item The class of $(-1)$-truncated maps in $\mathbb{C}$ is closed under composition, right cancellation and pullback 
along fibrations. The class of $(-1)$-truncated fibrations is pullback-stable. 
\item For any map $f\colon A\rightarrow B$ in $\mathbb{C}$, the lift $l$ in (\ref{diagdefhtyimage}) is unique up to 
homotopy whenever it exists.
\item If $f=pj$ is a $(-1)$-truncation of $f$, and we are given another factorization $f=p\sprime j\sprime$ for a fibration 
$p\sprime$ together with a lift
\[\xymatrix{
  & E\ar@{->>}[dr]^p & \\
A\ar[dr]_{j\sprime}\ar[ur]^j &  & B \\
  & E\sprime\ar@{->>}[ur]_{p\sprime}\ar[uu]^l & 
}
\]
such that $p\circ l\sim p\sprime$, then $l$ is a homotopy-equivalence if and only if $f=p\sprime j\sprime$ is a $(-1)$-truncation as 
well. In other words, $(-1)$-truncations are unique up to, and invariant under, homotopy-equivalence whenever they exist.
\end{enumerate}
\end{corollary}
\begin{proof}
Part 1 is \cite[Corollary 3.7.14]{joyaltribes}.
For Part 2, suppose we are given two lifts $l_1$ and $l_2$ as in (\ref{diagdefhtyimage}). By assumption, we obtain 
homotopies $H_1\colon E\rightarrow PB$ and $H_2\colon E\rightarrow PB$ between $q\circ l_1$ and $p$, and $q\circ l_2$ 
and $p$, respectively. Concatenation of $H_1$ with the inversion of $H_2$ (see \cite[3]{shulmaninv} for the explicit 
constructions) yields a homotopy $H_1\ast H_2^{-1}$ between $q\circ l_1$ and $q\circ l_2$. Then Lemma~\ref{lemmamono} 
applied to the map
\[\xymatrix{	
E\ar@/_1pc/[ddr]_{(l_1,l_2)}\ar@/^1pc/[drr]^{H_1\ast H_2^{-1}}\ar[dr] & & & \\
 & PB\times_{B\times B}(C\times C)\ar[r]\ar@{->>}[d]\ar@{}[dr]|(.3){\pbs} & PB\ar@{->>}[d] \\
 & C\times C\ar[r]_{q\times q} & B\times B
}\]
yields a homotopy between $l_1$ and $l_2$ since $q$ is $(-1)$-truncated.

For Part 3, suppose that both $A\xrightarrow{j}E\xrightarrow{p} B$ and $A\xrightarrow{j\sprime}E\sprime\xrightarrow{p\sprime} B$ are 
$(-1)$-truncations of the same map $f\colon A\rightarrow B$. Then we obtain lifts $l\colon E\rightarrow E\sprime$ and 
$l\sprime\colon E\sprime\rightarrow E$ as in Definition~\ref{defhtyimage}. By assumption, the triangles in (\ref{diagdefhtyimage}) 
commute up to homotopy, and so we may apply Part 2 to the pairs $(l\sprime\circ l,1_E)$ and $(l\circ l\sprime, 1_{E\sprime})$ together 
with the respectively concatenated homotopies into $B$ to obtain homotopies $l\sprime\circ l\sim 1_E$ and
$l\circ l\sprime\sim  1_{E\sprime}$.

For the other direction, it suffices to show that the fibration $p\sprime$ is $(-1)$-truncated, existence of the lifts in 
(\ref{diagdefhtyimage}) then follows by precomposition with $l$. But $(-1)$-truncatedness is invariant under composition with 
homotopy-equivalences, see e.g.\ \cite[Proposition 3.7.13]{joyaltribes}.
\end{proof}

Dually, a map $f$ in $\mathbb{C}$ is $(-1)$-connected if for every $(-1)$-truncation $f=j p$ of $f$, the fibration $p$ is 
an acyclic fibration.

\begin{corollary}\label{corhtyimage2}
Let $\mathbb{C}$ be a type theoretic fibration category.
\begin{enumerate}
\item For any map $f\colon A\rightarrow B$ in $\mathbb{C}$, the left map $j$ in a $(-1)$-truncation $f=pj$ of $f$ is $(-1)$-connected 
(whenever it exists).
\item A map in $\mathbb{C}$ is both $(-1)$-connected and $(-1)$-truncated if and only if it is a homotopy-equivalence.
\end{enumerate}
\end{corollary}
\begin{proof}
For Part 1, let $f=pj$ be a $(-1)$-truncation in $\mathbb{C}$, and $j=qi$ be a $(-1)$-truncation of $j\colon A\rightarrow E$ in
$\mathbb{C}$. We obtain the following diagram.
\[\xymatrix{
 & E\sprime\ar[dd]^{q}\ar@{->>}[dr]^{q\circ p} & \\
A\ar[ur]^{i}\ar[dr]_{j} & & B\\
 & E\ar@{->>}[ur]_{p} & 
}\]
Both top and bottom composition are a $(-1)$-truncation of $f\colon A\rightarrow B$ by
Corollary~\ref{corhtyimage1}.1, and so by Corollary~\ref{corhtyimage1}.3 there is a homotopy-equivalence
$e\colon E\sprime \rightarrow E$ making the resulting right hand side triangle commute up to homotopy. By
Corollary~\ref{corhtyimage1}.2, the lifts $q$ and $e$ are homotopic, and so $q$ is a homotopy-equivalence as well. Thus, 
$j$ is $(-1)$-connected.

For Part 2, if $f\colon A\rightarrow B$ is an acyclic fibration then $P_BA\twoheadrightarrow A\times_B A$ is an acyclic fibration as 
well. In particular, if $f$ is a homotopy-equivalence, it 
is $(-1)$-truncated as well. Then the factorization of $f$ into an acyclic cofibration $j$ followed by an acyclic
fibration $p$ is a $(-1)$-truncation of $f$. Indeed, the lifts (\ref{diagdefhtyimage}) are obtained via any 
chosen retract of $j$ (constructed in \cite[Lemma 3.6]{shulmaninv}). Thus, $f$ is $(-1)$-connected (as any two $(-1)$-truncations are 
homotopy-equivalent to one another by Corollary~\ref{corhtyimage1}.2).

Vice versa, let $f=pj$ be a factorization of $f$ into an acyclic cofibration $j$ followed by a fibration $p$. If
$f\colon A\rightarrow B$ is $(-1)$-truncated, then so is $p$, and thus $f=pj$ is a $(-1)$-truncation of $f$ 
again via \cite[Lemma 3.6]{shulmaninv}. If $f$ is furthermore $(-1)$-connected, then $p$ is an acyclic fibration, and so 
$f=pj$ is a homotopy-equivalence.
\end{proof}

The following definition will be the basis for the rest of this section.

\begin{definition}\label{defpbstablehtyimagefac}
A type theoretic fibration category $\mathbb{C}$ has \emph{$(-1)$-truncations}, if every map $f\colon A\rightarrow B$ in
$\mathbb{C}$ has a $(-1)$-truncation $f=pj$, and furthermore, for any fibration $q\colon F\twoheadrightarrow B$, the induced 
factorization $q^{\ast}f=q^{\ast}p\circ q^{\ast}j$ is a $(-1)$-truncation of $q^{\ast}f$.
\end{definition}

A type theoretic fibration category $\mathbb{C}$ has $(-1)$-truncations whenever the internal type theory of $\mathbb{C}$ comes 
equipped with propositional truncation as in \cite[Section 3.7]{hott}. That is the case for instance when $\mathbb{C}=\mathbb{M}^f$ 
for a combinatorial model category $\mathbb{M}$ whose cofibrations are the monomorphisms. Indeed, as presented for example in 
\cite[Section 7]{rezkhtytps}, in this case the $(-1)$-connected cofibrations and the $(-1)$-truncated fibrations form a weak 
factorization system on $\mathbb{M}$.
 
\begin{lemma}\label{lemmacnnctprops}
Suppose $\mathbb{C}$ has $(-1)$-truncations. 
\begin{enumerate}
\item The class of $(-1)$-connected maps in $\mathbb{C}$ is closed under pullbacks along fibrations. 
\item If $f\colon A\rightarrow B$ is $(-1)$-connected, then any factorization $f=qi$ for $q$ a $(-1)$-truncated 
fibration is a $(-1)$-truncation. In particular, $q$ is an acyclic fibration.
\item Every factorization of a map $f$ into a $(-1)$-connected map followed by a $(-1)$-truncated fibration is a $(-1)$-truncation of 
$f$.
\item The class of $(-1)$-connected maps is closed under composition.
\item The class of $(-1)$-connected maps is closed under  homotopy-equivalence. That means given a square of the form
\begin{align}\label{diaglemmacnnctprops}
\begin{gathered}
\xymatrix{
X\ar[d]_g\ar[r]_{\simeq}^e & A\ar[d]^f\\
Y\ar[r]_{\simeq}^d & B,
}
\end{gathered}
\end{align}
then $g$ is $(-1)$-connected if and only if $f$ is so.
\end{enumerate}
\end{lemma}
\begin{proof}
For Part 1, let $f\colon A\rightarrow B$ be $(-1)$-connected and $p\colon E\twoheadrightarrow B$ be a fibration. Let 
$f=qj$ be a homotopy-image factorization of $f$, so $q$ is an acyclic fibration. By assumption, the pullback
$p^{\ast}q\circ p^{\ast}j$ is a $(-1)$-truncation of $p^{\ast}f$. Since acyclic fibrations are closed under 
pullback, and $(-1)$-truncations are unique up to homotopy-equivalence by Corollary~\ref{corhtyimage1}.3, it 
follows that the map $p^{\ast}f$ is $(-1)$-connected again.

For Part 2, let $f\colon A\rightarrow B$ be $(-1)$-connected and $f=pj$ its $(-1)$-truncation through some 
object $E$, and $f=qi$ be another factorization through a $(-1)$-truncated fibration $q\colon C\twoheadrightarrow B$. By 
assumption, we obtain a map $l\colon E\rightarrow C$ together with a homotopy $q\circ l\sim p$ in $B$. Since $p$ is a 
acyclic fibration, we obtain a map $k\colon C\rightarrow E$ over $B$ as well. Both $p$ and $q$ are $(-1)$-truncated, and 
so $k$ and $l$ are mutual homotopy-inverses. Thus, $f=qi$ is a $(-1)$-truncation as well by Corollary~\ref{corhtyimage1}.4.

For Part 3, let $f=pj$ such that $p$ is a $(-1)$-truncated fibration and $j$ is $(-1)$-connected. Let $f=qi$ such that
$q\colon C\rightarrow B$ is a $(-1)$-truncated morphism as well. Consider the pullback
\[\xymatrix{
 & & E\ar@{->>}[dr]^p & \\
A\ar@/^1pc/[urr]^{j}\ar@/_1pc/[drr]_{i}\ar[r] & P\ar[ur]\ar@{->>}[dr]\ar@{}[rr]|(.3){\rotatebox[origin=c]{45}{$\pbs$}} & & B \\
 & & C\ar[ur]_{q} & 
}\]
If we factor the map $q$ into an acyclic cofibration $k$ followed by a $(-1)$-truncated fibration $q\sprime$, we see 
that the map $P\rightarrow E$ is the composition of the acyclic cofibration $p^{\ast}k$ and the fibration
$p^{\ast}q\sprime$. The latter is $(-1)$-truncated by Corollary~\ref{corhtyimage1}.1. Since $j$ is
$(-1)$-connected, it follows from Part 2 that $p^{\ast}q\sprime$ is an acyclic fibration. Postcomposition of a
homotopy-inverse of the resulting homotopy-equivalence $p^{\ast}q$ with $P\twoheadrightarrow C$ yields a lift as 
required in Definition~\ref{defhtyimage}.2.

For Part 4, let $f\colon A\rightarrow B$ and $g\colon B\rightarrow C$ be $(-1)$-connected morphisms and let $gf=pj$ be a 
$(-1)$-truncation of their composition. Consider the following pullback.
\[\xymatrix{
  &   &   E\ar@/^1pc/@{->>}[ddr]^p & \\
  & P\ar[ur]\ar@{->>}[dr]\ar@{}[rr]|(.2){\rotatebox[origin=c]{45}{$\pbs$}} &   &  \\
A\ar[rr]_f\ar@/^2pc/[uurr]^{j}\ar[ur]& & B\ar[r]_g & C 
}\]
As a pullback of $p$, the fibration $P\twoheadrightarrow B$ is $(-1)$-truncated. Since $f$ is $(-1)$-connected, it 
follows that $P\twoheadrightarrow B$ is an acyclic fibration by Part 2. Choosing a section
$s\colon B\rightarrow P$, we obtain a factorization of $g$ through $p$. Since $g$ is $(-1)$-connected as well, we see 
that $p$ is an acyclic fibration by the same reasoning. Thus, the composite $gf$ is $(-1)$-connected (using that
$(-1)$-truncations of $gf$ are unique up to homotopy-equivalence by Corollary~\ref{corhtyimage1}.4).

For Part 5, suppose we are given maps $f\colon A\rightarrow B$ and $g\colon X\rightarrow Y$ together with homotopy-equivalences as in 
(\ref{diaglemmacnnctprops}). If $g$ is $(-1)$-connected, let $f=pj$ be a $(-1)$-truncation
of $f$. Then $j$ is $(-1)$-connected by Corollary~\ref{corhtyimage2}.1, and hence so is the composition $je$ by 
Corollary~\ref{corhtyimage2}.2 and Part 4. In the same way, the composition $dg$ is $(-1)$-connected. Thus, via Part 3, 
$dg=p(je)$ is a $(-1)$-truncation of a $(-1)$-connected map. It follows that $p$ is an acyclic fibration, and 
hence $f$ is $(-1)$-connected (again via Corollary~\ref{corhtyimage1}.3).

Vice versa, suppose $f$ is $(-1)$-connected and factor the bottom homotopy-equivalence $d\colon Y\rightarrow B$ into an 
acyclic cofibration $j\colon Y\rightarrow C$ followed by an acyclic fibration $p\colon C\rightarrow B$.
\[\xymatrix{
X\ar[d]_g\ar[r]\ar@/^1pc/[rr]^e & P\ar@{->>}[r]^{\sim}\ar[d]\ar@{}[dr]|(.3){\pbs} & A\ar[d]^f \\
Y\ar@{^(->}[r]_j^{\sim} & C\ar@{->>}[r]^{\sim}_p & B  
}\]
Then the map $X\rightarrow P$ is a homotopy-equivalence  by 2-out-of-3, and the map $P\rightarrow C$ is $(-1)$-connected 
by Part 1. It follows that the composition $X\rightarrow C$ is $(-1)$-connected. The acyclic cofibration
$j\colon Y\rightarrow C$ admits a retract $r\colon C\rightarrow Y$ by \cite[Lemma 3.6]{shulmaninv} which is a homotopy-equivalence 
itself. Thus, $g$ is the composition of the $(-1)$-connected map $X\rightarrow C$ with the homotopy-equivalence
$r\colon C\rightarrow Y$, and hence $(-1)$-connected.
\end{proof}

Lastly, $(-1)$-truncations are homotopy-invariant in the following way.

\begin{lemma}\label{lemmahmtyimageinvariance}
Suppose $\mathbb{C}$ has $(-1)$-truncations. Let $f_1=pj$ be a $(-1)$-truncation of $f_1\colon A\rightarrow B$ in $\mathbb{C}$, and 
suppose $f_1$ is homotopic to $f_2\colon A\rightarrow B$. Then for every factorization $f_2=qi$ such that
\begin{enumerate}
\item the map $q\colon C\rightarrow B$ is $(-1)$-truncated, there is a vertical map
\[\xymatrix{
  & E\ar@{->>}[dr]^p\ar[dd]^l  & \\
A\ar[dr]_{i}\ar[ur]^j &  & B \\
  & C\ar[ur]_{q}& 
}
\]
such that both triangles commute up to homotopy;
\item the map $i\colon A\rightarrow C$ is $(-1)$-connected, there is a vertical map
\begin{align}\label{diaglemmahmtyimageinvariancepart2}
\begin{gathered}
\xymatrix{
  & E\ar@{->>}[dr]^p  & \\
A\ar[dr]_{i}\ar[ur]^j &  & B \\
  & C\ar[ur]_{q}\ar[uu]^l& 
}
\end{gathered}
\end{align}
such that both triangles commute up to homotopy;
\item the map $i\colon A\rightarrow C$ is $(-1)$-connected and $q\colon C\rightarrow B$ is $(-1)$-truncated, the lifts 
$l$ from Parts 1 and 2 are mutual homotopy-inverses.
\end{enumerate}
\end{lemma}
\begin{proof}
Let $f_1=pj$ be a $(-1)$-truncation of $f_1$ and let $H\colon A\rightarrow PB$ be a homotopy from $f_1$ to 
$f_2$. Choose a $(-1)$-truncation $H=rk$. We obtain a diagram of solid arrows
\begin{align}\label{diaglemmahmtyimageinvariance}
\begin{gathered}
\xymatrix{
 & E\ar@{->>}@/^1pc/[drr]^p\ar[d]^{l_1}_{\rotatebox[origin=c]{90}{$\simeq$}} & &  \\
A\ar@/^1pc/[ur]^j\ar@/_1pc/[dr]_i\ar[r]_k & D\ar@{--}[d]^{l_2}\ar@{->>}[r]_r & PB\ar@<.5ex>@{->>}[r]^{s}\ar@<-.5ex>@{->>}[r]_{t} & B  \\
  & C,\ar@/_1pc/[urr]_q & &  
}
\end{gathered}
\end{align}
together with homotopies $p\sim srl_1$ and $s\sim t$. The upper homotopy-equivalence $l_1$ exists because 
the composition $D\overset{r}{\twoheadrightarrow}PB\overset{s}{\twoheadrightarrow}B$ is a $(-1)$-truncated fibration,
$k$ is $(-1)$-connected and the upper triangle commutes. 

In Part 1, we obtain a vertical dashed arrow $l_2\colon D\rightarrow C$ together with a homotopy $tr\sim ql_2$, because the lower 
triangle commutes likewise. We obtain a composite $l\colon E\rightarrow C$ together with a homotopy from $p$ to
$q\circ l$ by concatenation of the three induced homotopies from $p$ to $srl_1$, from $srl_1$ to $trl_1$, and from 
$trl_1$ to $ql_2l_1$. Homotopy-commutativity of the left hand side triangle again follows automatically.

In Part 2, the lower triangle in (\ref{diaglemmahmtyimageinvariance}) yields a strictly commutative square of the form
\begin{align}\label{diaglemmahmtyimageinvariance2}
\begin{gathered}
\xymatrix{
A\ar[r]^k\ar[d]_i & D\ar@{->>}[d]^{tr} \\
C\ar[r]_q & 	B. \\
}
\end{gathered}
\end{align}

The natural map $A\rightarrow C\times_B D$ is a homotopy-equivalence by pullback stability of $(-1)$-truncated fibrations 
(Lemma~\ref{corhtyimage1}.1) and Lemma~\ref{lemmacnnctprops}.2.	 We hence obtain a lift $l_2\colon C\rightarrow D$ in 
(\ref{diaglemmahmtyimageinvariance}) such that the resulting triangles commute up to homotopy. We eventually obtain a composite lift 
$l:=l_1\circ l_2$ in (\ref{diaglemmahmtyimageinvariancepart2}) as required.

In Part 3, if $q$ is furthermore $(-1)$-truncated, any factorization of $q$ into an acyclic cofibration $m\colon C\rightarrow F$ 
followed by a $(-1)$-truncated fibration $t\colon F\twoheadrightarrow B$ yields another $(-1)$-truncation
$f_2=t(mi)$. We obtain a homotopy-equivalence $D\rightarrow F$ which may be composed with any retract of
$m\colon C\rightarrow F$ to yield a composite homotopy-equivalence $D\rightarrow C$ (which is homotopic to $l_2$ by an 
application of Corollary~\ref{corhtyimage1}.2). The fact that the lifts from Part 1 and Part 2 are mutual homotopy-inverses now 
follows again from the fact that $p$ and $q$ are $(-1)$-truncated.
\end{proof}

\begin{remark}
Whenever $\mathbb{C}$ has $(-1)$-truncations, the class of $(-1)$-connected maps and the 
class of $(-1)$-truncated maps form a suitably defined notion of homotopy factorization system on $\mathbb{C}$.
\end{remark}

\subsection{Univalent completion of fibrations}\label{secsubunivcompl}
Given a type theoretic fibration category $\mathbb{C}$, let $\mathcal{F}\subseteq\mathbb{C}^{[1]}$ be the full 
subcategory of fibrations in $\mathbb{C}$, and let $\mathcal{F}^{\times}\subset\mathbb{C}^{[1]}$ the subcategory whose 
objects are fibrations and whose arrows are homotopy-cartesian squares between fibrations in $\mathbb{C}$. That is, 
squares of the form
\begin{equation}\label{genfibsquare}
\begin{gathered}
\xymatrix{
X\ar[r]\ar@{->>}[d]_q& E\ar@{->>}[d]^{p} \\
Y\ar[r]_{b} & B 
}
\end{gathered} 
\end{equation}
such that the natural gap map $X\rightarrow E\times_B Y$ is a homotopy-equivalence.
\begin{definition}
A homotopy-cartesian square (\ref{genfibsquare}) of fibrations in $\mathbb{C}$ is a \emph{BM-equivalence} if the base 
map $b\colon Y\rightarrow B$ is $(-1)$-connected.
\end{definition}

For the following, we fix a fibration $\pi\colon \tilde{U}\twoheadrightarrow U$ in $\mathbb{C}$. We say that a 
fibration $p\colon E\twoheadrightarrow B$ is \emph{small} if it arises as the homotopy-pullback of $\pi$ along some map
$B\rightarrow U$. In the spirit of \cite{univalentcompletion}, we make the following definition.

\begin{definition}\label{defunicompletion}
Let $p\colon E\twoheadrightarrow B$ be a small fibration in $\mathbb{C}$. We say that a BM-equivalence
\[\xymatrix{
E\ar[r]\ar@{->>}[d]_p& u(E)\ar@{->>}[d]^{u(p)} \\
B\ar[r]_{\iota} & u(B) 
}\]
is a \emph{univalent completion} of $p$ if the fibration $u(p)\in\mathbb{C}$ is small and univalent.
\end{definition}

The authors of \cite{univalentcompletion} show that in the case of simplicial sets, a 
univalent completion exists for all $\kappa$-small Kan fibrations between Kan complexes for sufficiently large cardinals 
$\kappa$. In our terms, this is a univalent completion with respect to the univalent fibration
$\pi_{\kappa}\colon\tilde{U}_{\kappa}\twoheadrightarrow U_{\kappa}$ that is universal for all $\kappa$-small Kan 
fibrations, shown to exist in \cite{klvsimp}. The fact that the map $\iota\colon B\rightarrow u(B)$ constructed in 
\cite[Theorem 5.1]{univalentcompletion} is $(-1)$-connected is not explicitly noted in \cite[Theorem 5.1]
{univalentcompletion}, but follows immediately since
$\iota_0\colon B_0\rightarrow u(B)_0$ is the identity on the sets of vertices. Given the following syntactic 
generalization of Gepner and Kock's observation in \cite[Corollary 3.10]{gepnerkock})\footnote{Also generally shown for 
locally cartesian closed $(\infty,1)$-categories in \cite[Theorem 6.29]{rasekh} without the assumption of 
presentability.}, it is easy to show that univalent completion always exists whenever the ambient fibration
$\pi\colon\tilde{U}\twoheadrightarrow U$ is univalent itself.

\begin{proposition}\label{propmereproppb}
Let $p\colon E\twoheadrightarrow B$ and $q\colon X\twoheadrightarrow Y$ be fibrations in a type theoretic fibration 
category $\mathbb{C}$. Suppose $p$ is univalent and  
\begin{align}\label{diagmereproppb}
\begin{gathered}
\xymatrix{
X\ar[r]\ar@{->>}[d]_q & E\ar@{->>}[d]^p \\
Y\ar[r]_f & B
}
\end{gathered}
\end{align}
is a homotopy-cartesian square. Then the fibration $q$ is univalent if and only if the map $f$ is $(-1)$-truncated.
\end{proposition}
\begin{proof}
Factoring $f$ into an acyclic cofibration $j\colon Y\overset{\sim}{\hookrightarrow}Z$ followed by a fibration
$m\colon Z\twoheadrightarrow E$, it suffices to show that $m$ is $(-1)$-truncated if and only if the fibration
$m^{\ast}p$ is univalent, since univalence is invariant under homotopy-equivalence of fibrations by
\cite[Corollary 1.6.4]{thesis}. Thus, without loss of generality we may assume the square~(\ref{diagmereproppb}) is 
strictly cartesian and $f$ is a fibration. We obtain the following diagram.
\[\xymatrix{
Y\ar[dr]\ar@/_1pc/[dddr]_{\Delta}\ar[rrr]^f & & & B\ar@{^(->}[dr]^{\sim}_{r_B}\ar@/_1pc/[dddr]_{\Delta}|(.3)\hole|(.63)\hole & & & \\
 & (f\times f)^{\ast}PB\ar[dr]\ar@{->>}[dd]\ar[rrr]\ar@{}[ddrr]_(.2){\pbs} & & & PB\ar@{->>}[dd]|(.5)\hole\ar[dr]^{} & \\
 & & \mathrm{Equiv}(q)\ar@/^/@{->>}[dl]\ar[rrr]\ar@{}[drr]_(.2){\pbs} & & & \mathrm{Equiv}(p)\ar@/^/@{->>}[dl] \\
 & Y\times Y\ar[rrr]_{f\times f} & & & B\times B & 
}\]
Here, $\mathrm{Equiv}(p)\twoheadrightarrow B\times B$ is the object (\ref{diagequivdeg}) defined generally for Segal 
objects in $\mathbb{C}$. Assume $f$ is $(-1)$-truncated. Then the natural map $PY\rightarrow(f\times f)^{\ast}PB$ is a 
homotopy-equivalence by Lemma~\ref{lemmamono} and hence the composition $Y\rightarrow (f\times f)^{\ast}PB$ is a 
homotopy-equivalence. The fibration $p$ is univalent, i.e.\ the map
$PB\rightarrow\mathrm{Equiv}(p)$ is a homotopy-equivalence and hence so is its pullback
$(f\times f)^{\ast}PB\rightarrow\mathrm{Equiv}(q)$. Thus, the composition $Y\rightarrow\mathrm{Equiv}(q)$ is a homotopy
equivalence. Hence, $q$ is univalent.

Vice versa, assume the fibration $q$ is univalent. Then $PY\rightarrow\mathrm{Equiv}(q)$ is a 
homotopy-equivalence by univalence of $q$ and also $(f\times f)^{\ast}PB\rightarrow\mathrm{Equiv}(q)$ is a homotopy-equivalence by 
univalence of $p$. So we see that $PY\simeq(f\times f)^{\ast}PB$ over $Y\times Y$ and so the fibration $f\colon Y\twoheadrightarrow B$ 
is $(-1)$-truncated by Lemma~\ref{lemmamono}.
\end{proof}

\begin{proposition}\label{propexunivcomp}
Let $\mathbb{C}$ be a type theoretic fibration category with $(-1)$-truncations, and suppose $\pi\colon\tilde{U}\twoheadrightarrow U$ 
in $\mathbb{C}$ is univalent. Then every small fibration $p\colon E\twoheadrightarrow B$ in $\mathbb{C}$ has a univalent completion
\begin{align}\label{diagunivcompletion}
\begin{gathered}
\xymatrix{
E\ar[r]\ar@{->>}[d]_p& u(E)\ar@{->>}[d]^{u(p)} \\
B\ar[r]_{\iota} & u(B).
}
\end{gathered}
\end{align}
Furthermore, univalent completion is unique up to homotopy. That means, whenever
\begin{align}\label{diagunivcompletion2}
\xymatrix{
E\ar[r]\ar@{->>}[d]_p& u\sprime(E)\ar@{->>}[d]^{u\sprime(p)} \\
B\ar[r]_{\iota\sprime} & u\sprime(B)
}
\end{align}
is another univalent completion of $p$, there is a homotopy-commutative diagram of the form
\begin{align}\label{diagpropexunivcomp}
\begin{gathered}
\xymatrix{
  & u(E)\ar@{->>}[dd]|\hole^(.7){u\sprime(p)}\ar[dr]^{\simeq} & \\
E\ar@{->>}[dd]_p\ar[rr]\ar[ur] &  &  u\sprime(E)\ar@{->>}[dd]^(.3){u(p)} \\
 & u(B)\ar[dr]^{\simeq} &  \\
B\ar[rr]_{\iota\sprime}\ar[ur]_{\iota} & & u\sprime(B). 
}
\end{gathered}
\end{align}
\end{proposition}

\begin{proof}
Since $p\colon E\twoheadrightarrow B$ is small, there is a homotopy-cartesian square
\[\xymatrix{
E\ar[r]\ar@{->>}[d]_p& \tilde{U}\ar@{->>}[d]^{\pi} \\
B\ar[r]_{b} & U
}\]
in $\mathbb{M}$. By Corollary~\ref{corhtyimage2}.1, we can factor the map $b\colon B\rightarrow U$ into a
$(-1)$-connected map $\iota\colon B\rightarrow u(B)$ followed by a $(-1)$-truncated 
fibration $b_{-1}\colon u(Y)\twoheadrightarrow B$ and obtain two homotopy-cartesian squares
\begin{align}\label{diagunivcompletionexist}
\begin{gathered}
\xymatrix{
E\ar[r]\ar@{->>}[d]_p&(b_{-1})^{\ast}\tilde{U}\ar@{->>}[d]_{(b_{-1})^{\ast}\pi}\ar[r]\ar@{}[dr]|(.3){\pbs} & \tilde{U}\ar@{->>}[d]^{\pi} \\
B\ar[r]_(.4){\iota} & u(B)\ar@{->>}[r]_{b_{-1}} & U.
}
\end{gathered}
\end{align}
But $\pi$ is univalent in $\mathbb{C}$ and $b_{-1}\colon u(B)\twoheadrightarrow U$ is $(-1)$-truncated, and so the 
fibration $(b_{-1})^{\ast}p$ is univalent, too, by Proposition~\ref{propmereproppb}. We therefore can define
$u(p):=(b_{-1})^{\ast}\pi$.

Given another univalent completion (\ref{diagunivcompletion2}) of $p$, we obtain two sequences of squares between 
fibrations of the form
\begin{align}\label{diagunivcompletionunique}
\begin{gathered}
\xymatrix{
  & u(E)\ar[rr]\ar@{->>}[dd]|\hole_(.7){u(p)} &       & \tilde{U}\ar@{->>}[dd]^{\pi} \\
E\ar@{->>}[dd]_p\ar[rr]\ar[ur] &             &  u\sprime(E)\ar@{->>}[dd]^(.3){u\sprime(p)}\ar[ur] &            \\
  & u(B)\ar@{->>}[rr]|(.55)\hole_(.3){b_{-1}} &       &     U      \\
B\ar[rr]_{\iota\sprime}\ar[ur]_{\iota} &             &  u\sprime(B)\ar[ur]_{b\sprime_{-1}} & 
}
\end{gathered}
\end{align}
where all sides but potentially the top and bottom side commute and are homotopy-cartesian.
If we denote by $b\sprime$ the composition $b\sprime_{-1}\circ\iota\sprime\colon B\rightarrow U$, the lift
\[\xymatrix{
 & \mathrm{Equiv}(\pi)\ar@{->>}[d]^{(s,t)} \\
B\ar[r]_{(b,b\sprime)}\ar[ur]^{(b,b\sprime,\mathrm{id}_p)} & U\times U
}\]
is a homotopy between $b$ and $b\sprime$ by univalence of $\pi$. In other words, the bottom square of diagram 
(\ref{diagunivcompletionunique}) commutes up to homotopy. Since both $\iota$ and $\iota\sprime$ are $(-1)$-connected by 
assumption, and both $b_{-1}$ and $b\sprime_{-1}$ are $(-1)$-truncated by Proposition~\ref{propmereproppb}, we obtain a 
homotopy-equivalence $u(B)\rightarrow u\sprime (B)$ such that both resulting triangles commute up to homotopy by 
Lemma~\ref{lemmahmtyimageinvariance}.3. Since both $(-1)$-connectedness and $(-1)$-truncatedness are stable by pullback 
along fibrations (Corollary~\ref{corhtyimage1}.1, Lemma~\ref{lemmacnnctprops}.1), the same reasoning applies to the top 
square, and we obtain a homotopy-equivalence $u(E)\rightarrow u\sprime(E)$ making the resulting triangles commute up to 
homotopy. The resulting square
\[\xymatrix{
u(E)\ar[r]^{\simeq}\ar@{->>}[d]_{u(p)} & u\sprime(E)\ar@{->>}[d]^{u\sprime(p)} \\
u(B)\ar[r]_{\simeq} & u\sprime(B)}\]
commutes up to homotopy via Lemma~\ref{lemmamono}, because it does so after postcomposition with the $(-1)$-truncated map 
$b\sprime_{-1}$.
\end{proof}

Let $\mathcal{UF}\subseteq\mathcal{F}$ denote the full subcategory of univalent fibrations in $\mathbb{C}$, and
$\mathcal{W}$ denote both the class of homotopy-equivalences in $\mathbb{C}$ and the class of pointwise weak
equivalences in $\mathcal{F}$, respectively. Let BM be the class of all BM-equivalences in $\mathbb{C}$.
Recall that a map of relative categories $F\colon(\mathbb{C},W)\rightarrow(\mathbb{D},V)$ is a \emph{relative inclusion} 
if the underlying functor of categories $F\colon\mathbb{C}\rightarrow\mathbb{D}$ is a (not-necessarily full) inclusion and
$\mathbb{C}\cap V=W$.

\begin{corollary}\label{corufincbm}
The inclusion $\mathcal{UF}\subseteq\mathcal{F}$ induces a relative inclusion
\[(\mathcal{UF}^{\times},\mathcal{W})\hookrightarrow(\mathcal{F}^{\times},\mathrm{BM})\]
of relative categories.
\end{corollary}
\begin{proof}
Clearly every weak equivalence is a BM-equivalence, so we have to show that vice versa every BM-equivalence of the 
form (\ref{genfibsquare}) between univalent fibrations is a homotopy-equivalence in $\mathbb{C}$. But the base map
$b\colon Y\rightarrow B$ is $(-1)$-connected by assumption and it is $(-1)$-truncated by
Proposition~\ref{propmereproppb}. Thus $b$ is a homotopy-equivalence by Corollary~\ref{corhtyimage2}.2, and so is the 
upper map $X\rightarrow E$ since the square is homotopy-cartesian.
\end{proof}

\subsection{Rezk-completion of Segal objects}\label{secsubrezkcompl}

In virtue of Theorem~\ref{intc=simpc}, we will continue to consider Segal objects $X$ in type theoretic 
fibration categories $\mathbb{C}$, and define a ``Rezk-completion'' of $X$ to be a \emph{DK-equivalence} to a 
univalent Segal object in $\mathbb{C}$. Thus, Rezk-completion as considered here is in essence what is considered 
syntactically in \cite[8]{aksrezkcompl} for precategories. In this subsection we fix a type theoretic fibration category $\mathbb{C}$ 
with $(-1)$-truncations.

Let $\mathrm{S}(\mathbb{C})\subseteq s\mathbb{C}$ denote the full subcategory of Segal objects in $\mathbb{C}$. 
To define DK-equivalences and Rezk-completion in this context, consider the following pullbacks defined for any 
map $f\colon X\rightarrow Y$ between Segal objects.
\begin{align}\label{diagfcone}
\begin{gathered}
\xymatrix{
Mf\ar[r]\ar@{->>}[d]\ar@{}[dr]|(.3){\pbs} & \mathrm{Equiv}Y\ar@{->>}[d] \\
f\downarrow Y\ar@{->>}[d]\ar[r]\ar@{}[dr]|(.3){\pbs} & Y_1\ar@{->>}[d]^{(s,t)} \\
X_0 \times Y_0\ar[r]^{f_0\times 1}\ar@{->>}[d]_{\pi_1}\ar@{}[dr]|(.3){\pbs} & Y_0\times Y_0\ar@{->>}[d]^{\pi_1} \\
X_0\ar[r]_{f_0} & Y_0
}
\end{gathered}
\end{align}

\begin{definition}\label{defDKequiv}
Let $f\colon X\rightarrow Y$ be a map between Segal objects in $\mathbb{C}$. We say that
\begin{enumerate}
\item $f$ is \emph{fully faithful} if the natural gap map $\eta_f\colon X_1\rightarrow(f_0\times f_0)^{\ast}Y_1$ over 
$X_0\times X_0$ is a homotopy-equivalence.
\item $f$ is \emph{essentially surjective} if the composition
$Mf\twoheadrightarrow X_0\times Y_0\overset{\pi_2}{\twoheadrightarrow} Y_0$ is $(-1)$-connected.
\item $f$ is a \emph{DK-equivalence} if it is fully faithful and essentially surjective.
\end{enumerate}
\end{definition}
In the case $\mathbb{C}$ is the fibration category of Kan complexes, a map between Segal spaces is a DK-equivalence in 
the sense of Definition~\ref{defDKequiv} if and only if it is a Dwyer-Kan equivalence in the sense of
\cite[Section 7.4]{rezk}. Indeed, a map between Segal spaces $f\colon X\rightarrow Y$ is fully faithful, i.e.\ the map 
$\eta_f\colon X_1\rightarrow (f_0\times f_0)^{\ast}Y_1$ is a levelwise homotopy-equivalence, if and only if for all 
objects $x_1,x_2$ of $X$ presented as global elements $x_i\colon\Delta^0\rightarrow X_0$, the pullback
$(x_1,x_2)^{\ast}\eta_f\colon(x_1,x_2)^{\ast}X_1\rightarrow((f\times f)\circ(x_1,x_2))^{\ast}Y_1$ is a levelwise 
homotopy-equivalence. By definition, that is, if and only if the functors
$\mathrm{map}_f\colon\mathrm{map}_X(x_1,x_2)\rightarrow\mathrm{map}_Y(f_0(x_1),f_0(x_2))$ of mapping spaces as defined 
in \cite[5]{rezk} are homotopy-equivalences.
Similarly, one shows that $f\colon X\rightarrow Y$ is essentially surjective as defined in Definition~\ref{defDKequiv}.2 
if and only if the induced functor $\mathrm{Ho}(f)\colon\mathrm{Ho}(X)\rightarrow\mathrm{Ho}(Y)$ of associated 
homotopy categories is essentially surjective.

\begin{lemma}\label{lemmaDKequivinv}
All three notions in Definition~\ref{defDKequiv} are invariant under homotopy-equivalence between Segal objects. That 
means, given a square
\[\xymatrix{
W\ar[d]_f\ar[r] & X\ar[d]^g \\
Y\ar[r] & Z
}\]
between Segal objects such that the horizontal maps are levelwise homotopy-equivalences, then $f$ is fully faithful/
essentially surjective/a DK-equivalence if and only if $g$ is so.
\end{lemma}
\begin{proof}
Since homotopy-equivalences in $\mathbb{C}$ are preserved by pullback along fibrations, we obtain a homotopy-equivalence
\[\xymatrix{
W_1\ar[d]\ar[r]^{\simeq} & X_1\ar[d] \\
(f_0\times f_0)^{\ast}Y_1 \ar[r]_{\simeq} & (g_0\times g_0)^{\ast}Z_1
}\]
between the natural maps considered in Definition~\ref{defDKequiv}.1 associated to $f$ and $g$, respectively. It follows 
that $f$ is fully faithful if and only if $g$ is so by 2-out-of-3. Similarly, via Diagram (\ref{diagfcone}), we obtain a 
homotopy-equivalence of the form
\[\xymatrix{
Mf\ar@{->>}[d]\ar[r]^{\simeq} & Mg\ar@{->>}[d] \\
Y_0\ar[r]_{\simeq} & Z_0.
}\]
By Lemma~\ref{lemmacnnctprops}.5, $(-1)$-connectedness in $\mathbb{C}$ is invariant under homotopy-equivalence, and so the same 
follows for essential surjectivity in $\mathrm{S}(\mathbb{C})$. The case for DK-equivalences now follows immediately.
\end{proof}

Let $\mathrm{S}(\mathbb{C})^{\times}\subseteq\mathrm{S}(\mathbb{C})$ denote the subcategory of Segal objects and fully 
faithful maps. Let $\mathrm{US}(\mathbb{C})^{\times}\subseteq S(\mathbb{C})^{\times}$ denote the full 
subcategory of univalent Segal objects in $\mathbb{M}$.

Just as the BM-equivalences between univalent fibrations, the DK-equivalences between univalent Segal objects are 
exactly the levelwise homotopy-equivalences.

\begin{lemma}
Let $f\colon X\rightarrow Y$ be a map between univalent Segal objects in $\mathbb{C}$. The map $f$ is a DK-equivalence 
if and only if it is a levelwise homotopy-equivalence. In particular, every DK-equivalence between univalent Segal 
objects is in fact split surjective, so that the map $Mf\twoheadrightarrow Y_0$ is an acyclic fibration.
\end{lemma}
\begin{proof}
If $f$ is a levelwise homotopy-equivalence, the maps $f_0\times f_0\colon X_0\times X_0\rightarrow Y_0\times Y_0$ and 
$f_1\colon X_1\rightarrow Y_1$ are homotopy-equivalences. Hence, the square
\[\xymatrix{
X_1\ar@{->>}[d]\ar[r]^{f_1} & Y_1\ar@{->>}[d] \\
X_0\times X_0\ar[r]_{f_0\times f_0} & Y_0 \times Y_0 
}\]
is homotopy-cartesian and the map $f$ is fully faithful. Furthermore, from Diagram~(\ref{diagfcone}) we see that the top 
map $Mf\rightarrow\text{Equiv}Y$ is a homotopy-equivalence, and so is $\mathrm{Equiv}Y\twoheadrightarrow Y_0$ by 
univalence of $Y$. Thus, the same holds for their composition $Mf\twoheadrightarrow Y_0$ which was to show.

Vice versa, assume $f$ is a DK-equivalence. Since both $X$ and $Y$ are Segal objects, it suffices to show that $f_0$ and 
$f_1$ are homotopy-equivalences. If $f_0$ is a homotopy-equivalence, then so will be $f_0\times f_0$ as well as its 
pullback along the boundary fibration $Y_1\twoheadrightarrow Y_0\times Y_0$. Since $f$ is full faithful, it follows that 
$f_1$ is a homotopy-equivalence. We are therefore left to show that $f_0$ is both $(-1)$-truncated and $(-1)$-connected. 

In order to show $(-1)$-truncatedness, we note that fully faithfulness of $f$ implies that the map
\[\text{Equiv}X\rightarrow(f_0\times f_0)^{\ast}\text{Equiv} Y\]
is a homotopy-equivalence. But by univalence of $X$ and $Y$ this map is equivalent to the transport
$\mathrm{tr}_{f_0}\colon PX_0\rightarrow (f_0\times f_0)^{\ast}PY_0$, and thus $(-1)$-truncatedness of $f_0$ follows 
from Lemma~\ref{lemmamono}. Towards $(-1)$-connectedness, on the one hand, we note that 
by Diagram~(\ref{diagfcone}) and essential surjectivity of $f$, the natural map $Mf\rightarrow\text{Equiv}Y$ is
$(-1)$-connected whenever $Y$ is univalent. On the other hand, the natural map $\text{Equiv}X\rightarrow Mf$ is a 
homotopy-equivalence since both $\text{Equiv}X$ and $Mf$ are contractible over $X_0$. Thus, the composition
$\text{Equiv}X\rightarrow\text{Equiv}Y$ is both $(-1)$-connected and homotopy-equivalent to
$f_0\colon X_0\rightarrow Y_0$.
\end{proof}

Let $\mathcal{W}\subseteq s\mathbb{C}$ denote the class of levelwise homotopy-equivalences in $\mathbb{C}$.
 
\begin{corollary}\label{correlinclsegal}
The inclusion $\mathrm{US}(\mathbb{C})^{\times}\subseteq\mathrm{S}(\mathbb{C})^{\times}$ induces a relative inclusion
\[(\mathrm{US}(\mathbb{C})^{\times},\mathcal{W})\hookrightarrow(\mathrm{S}(\mathbb{C})^{\times},\mathrm{DK})\]
of relative categories.\qed
\end{corollary}

Following \cite[14]{rezk}, given a map $f\colon X\rightarrow Y$ of Segal objects in $\mathbb{C}$, we will say that $f$ 
is a \emph{Rezk-completion} whenever $f$ is a DK-equivalence and $Y$ is univalent.

\begin{remark}\label{remrezkcomplexples}
Rezk-completions exist in a plethora of model categorical examples. Generalizing Rezk's constructions, Barwick has shown 
in \cite{barwickthesis} that Reedy fibrant Segal objects and complete Segal objects are the fibrant objects in a model 
structure on $s\mathbb{M}$ whenever $\mathbb{M}$ is a simplicial and monoidal model category that is combinatorial and 
left-proper.
From his results it follows that the homotopy theory of complete Segal objects in $\mathbb{M}$ is the localization of 
the homotopy theory of Reedy fibrant Segal objects at the DK-equivalences. In particular, the complete Segal objects are 
exactly the DK-local Reedy fibrant Segal objects. 
It also follows that Rezk-completion of a Segal object is simply given by fibrant replacement in the model structure for 
complete Segal objects (and hence always exists). 
\end{remark}

\subsection{A comparison of associated homotopy theories}\label{secsubcomparecompls}

Let $\mathbb{C}$ again be a type theoretic fibration category with $(-1)$-truncations. The nerve 
construction of a fibration $p$ in $\mathbb{C}$ to a Segal object $N(p)$ as defined in (\ref{diagdefNp}) induces a relative 
functor of relative categories in the following way.

\begin{proposition}\label{propcompletion}
Every homotopy-cartesian square
\begin{align}
\begin{gathered}
\xymatrix{
E\ar[r]\ar@{->>}[d]_p& E\sprime\ar@{->>}[d]^{p\sprime} \\
B\ar[r]_{\iota} & B\sprime
}
\end{gathered}
\end{align}
between fibrations $p,p\sprime$ in $\mathbb{C}$ functorially induces a fully faithful map
\[N(\iota)\colon N(p)\rightarrow N(p\sprime)\]
between associated Segal objects in $\mathbb{C}$. Furthermore, the following hold.
\begin{enumerate}
\item If $\iota$ is $(-1)$-connected, then $N(\iota)$ is also essentially surjective.
\item Suppose $p\sprime$ is univalent. If $N(\iota)$ is essentially surjective, then $\iota$ is $(-1)$-connected.
\end{enumerate}
We thus obtain a relative functor of the form
\begin{align*}
N\colon (\mathcal{F}^{\times},\mathrm{BM})\rightarrow (\mathrm{S}(\mathbb{C})^{\times},\mathrm{DK}).
\end{align*}
\end{proposition}

\begin{proof}
Let $p\colon E\twoheadrightarrow B$ and $p\sprime\colon E\sprime\twoheadrightarrow B\sprime$ be fibrations in
$\mathbb{C}$ and $N(p)$ and $N(p\sprime)$ in $s\mathbb{C}$ their associated Segal objects. Let
\[\xymatrix{
E\ar[r]\ar@{->>}[d]_p& E\sprime\ar@{->>}[d]^{p\sprime} \\
B\ar[r]_{\iota} & B\sprime.
}\]
be a homotopy-cartesian square so that we obtain a homotopy-equivalence
\[\xymatrix{
\mathrm{Fun}(p)\ar@/^/[drr]\ar@{-->}[dr]^{\sim}_{}\ar@/_/@{->>}[ddr] & & \\
 & \mathrm{Fun}(\iota^{\ast}p\sprime)\ar[r]\ar@{}[dr]|(.3){\pbs}\ar@{->>}[d] & \mathrm{Fun}(p\sprime)\ar@{->>}[d]\\
 & B\times B\ar[r]_(.4){\iota\times\iota} & B\sprime\times B\sprime,
}\]
essentially by the proof of \cite[Proposition 1.6.1]{thesis}. But the dashed map in the diagram is exactly
\[N(\iota)_1\colon N(p)_1\rightarrow(N(\iota)_0\times N(\iota)_0)^{\ast}N(p\sprime)_1\]
and so the map $N(\iota)\colon N(p)\rightarrow N(p\sprime)$ is fully faithful. Functoriality follows from a straight-forward
computation.

For Part 1, note that since $PB\sprime$ is contractible over $B\sprime$, the idtoequiv-lift
$P(B\sprime)\rightarrow\mathrm{Equiv}(p\sprime)$ over $B\sprime$ shows that the fibration
$\mathrm{Equiv}(p\sprime)\twoheadrightarrow B\sprime$ is $(-1)$-connected. Suppose that
$\iota\colon B\rightarrow B\sprime$ is $(-1)$-connected as well. Since $(-1)$-connected maps in $\mathcal{C}$ are 
preserved by pullback along fibrations, the map
\[\iota\times 1\colon B\times B\sprime\rightarrow B\sprime\times B\sprime\]
is $(-1)$-connected and hence so are both of the following horizontal maps.
\[\xymatrix{
M(N(\iota))\ar@{->>}[d]\ar[r]\ar@{}[dr]_(.3){\pbs} & \mathrm{Equiv}(p\sprime)\ar@{->>}[d] \\
B\times B\sprime\ar[r]_{\iota\times 1}\ar@{->>}@/_/[dr]_{\pi_2} & B\sprime\times B\sprime\ar@{->>}[d]^{\pi_2} \\ 
 & B\sprime
}\]
The composition $M(N(\iota))\twoheadrightarrow B\sprime$ is hence a composition of the two $(-1)$-connected maps 
$M(N(\iota))\rightarrow \mathrm{Equiv}(p\sprime)$ and $\mathrm{Equiv}(p\sprime)\twoheadrightarrow B\sprime$ and as such 
$(-1)$-connected itself.

For Part 2, assume $p\sprime$ is univalent and $N(\iota)$ is essentially surjective. By univalence of $p\sprime$, 
the map $P(B\sprime)\rightarrow \text{Equiv}(p\sprime)$ is a homotopy-equivalence over $B\sprime\times B\sprime$, and 
hence so is its pullback
\[\xymatrix{
P(\iota)\ar[r]\ar@{->>}@/_/[dr] & M(N(\iota))\ar@{->>}[d] \\
& B\times B\sprime 
}\]
along $\iota\times 1\colon B\times B\sprime\rightarrow B\sprime\times B\sprime$.
By essential surjectivity of $N(\iota)$, the composite fibration $M(N(\iota))\twoheadrightarrow B\sprime$ is
$(-1)$-connected, and thus so is the composite fibration
$P(\iota)\twoheadrightarrow B\times B\sprime\overset{\pi_2}{\twoheadrightarrow}B\sprime$. Furthermore, the natural gap 
map
\[\xymatrix{
B\ar[dr]\ar@/^1pc/[drr]^{\mathrm{idtoequiv}_{B\sprime}\circ\iota}\ar@/_1pc/@{=}[dddr] & & \\
 & P(\iota)\ar@{->>}[d]\ar@{}[dr]|(.3){\pbs}\ar[r] & P(B\sprime)\ar@{->>}[d]\ar@{->>}@/^2pc/[dd]^{\rotatebox[origin=c]{90}{$\sim$}} \\
 & B\times B\sprime\ar@{->>}[d]^{\pi_1}\ar@{}[dr]|(.3){\pbs}\ar[r] & B\sprime\times B\sprime\ar@{->>}[d]^{\pi_1} \\
 & B\ar[r]_{\iota} & B\sprime 
}\]
is a homotopy-equivalence by 2-out-of-3. Thus, the composite
\[B\rightarrow P(\iota)\twoheadrightarrow B\times B\sprime \overset{\pi_2}{\twoheadrightarrow} B\sprime\]
is a factorization of $\iota\colon B\rightarrow B\sprime$ into a homotopy-equivalence followed by a $(-1)$-connected 
fibration. Thus, the map $\iota$ is $(-1)$-connected as well.
\end{proof}

\begin{corollary}\label{thm5sum1}
The nerve construction induces a pullback
\begin{equation}\label{mainsquare}
\begin{gathered}
\xymatrix{
(\mathcal{UF}^{\times},\mathcal{W})\ar[r]^(.45)N\ar@{^(->}[d]\ar@{}[dr]|(.3){\pbs} & (\mathrm{US}(\mathbb{C})^{\times},\mathcal{W})\ar@{^(->}[d]\\
(\mathcal{F}^{\times},\mathrm{BM})\ar[r]_(.45)N & (\mathrm{S}(\mathbb{C})^{\times},\mathrm{DK})
}
\end{gathered}
\end{equation}
of relative categories such that
\begin{enumerate}
\item the vertical arrows are relative inclusions;
\item the relative functor $N$ reflects DK-equivalences with univalent codomain to BM-equivalences with univalent 
codomain (and hence reflects Rezk-completion to univalent completion).
\end{enumerate}
\end{corollary}

\begin{proof}
The vertical two relative functors exist and are relative inclusions by Corollary~\ref{corufincbm} and 
Corollary~\ref{correlinclsegal}. The horizontal two relative functors exist by Proposition~\ref{propcompletion}. The resulting 
square is a pullback square by Lemma~\ref{lemmasimpunivalence}. Part 2 follows from the fact that the square is a pullback square 
and Proposition~\ref{propcompletion}.2.
\end{proof}

Whenever $\mathbb{C}$ arises as the category of fibrant objects of a logical model category $\mathbb{M}$ whose fibrant 
objects are cofibrant, we considered the Reedy fibrant nerve construction $X_p:=\mathbb{R}N(p)$ of fibrations $p$ in $\mathbb{C}$ in 
Section~\ref{secunivvscompl}.
In this case, let $\mathrm{S}_f(\mathbb{C})^{\times}\subset\mathrm{S}(\mathbb{C})^{\times}$ be the full subcategory of Reedy fibrant 
Segal objects, and $\mathrm{CS}(\mathbb{C})^{\times}\subset\mathrm{S}_f(\mathbb{C})^{\times}$ be the full subcategory of complete 
Segal objects. 

\begin{proposition}\label{thmcmpcompletion}
Suppose $\mathbb{C}$ is the category of fibrant objects associated to a logical model category whose fibrant objects are 
cofibrant. 
Then any functorial Reedy fibrant replacement functor $\mathbb{R}$ of simplicial 
objects in $\mathbb{C}$ induces pullback squares
\begin{equation}\label{mainsquare2}
\begin{gathered}
\xymatrix{
(\mathcal{UF}^{\times},\mathcal{W})\ar[r]^(.45)N\ar@{^(->}[d]\ar@{}[dr]|(.3){\pbs} & (\mathrm{US}(\mathbb{C})^{\times},\mathcal{W})\ar@{^(->}[d]\ar[r]^{\mathbb{R}}\ar@{}[dr]|(.3){\pbs} & (\mathrm{CS}(\mathbb{C})^{\times},\mathcal{W})\ar@{^(->}[d]\\
(\mathcal{F}^{\times},\mathrm{BM})\ar[r]_(.45)N & (\mathrm{S}(\mathbb{C})^{\times},\mathrm{DK})\ar[r]^{\mathbb{R}} & (\mathrm{S}(\mathbb{C})_f^{\times},\mathrm{DK})
}
\end{gathered}
\end{equation}
of relative categories such that
\begin{enumerate}
\item the vertical arrows are relative inclusions;
\item the relative functor $X:=\mathbb{R}\circ N$ reflects DK-equivalences with complete codomain to BM-equivalences with univalent 
codomain (and hence reflects Rezk-completion to univalent completion).
\end{enumerate}

\end{proposition}
\begin{proof}
The square on the left hand side of Diagram~(\ref{mainsquare2}) is covered by Corollary~\ref{thm5sum1}. The vertical arrow on the 
right hand side is a relative inclusion, because it is the restriction of a relative inclusion along two relative inclusions, and 
its underlying functor $\mathrm{CS}(\mathbb{C})^{\times}\hookrightarrow\mathrm{S}(\mathbb{C})_f^{\times}$ is an inclusion of 
categories.

The upper horizontal arrow $\mathbb{R}$ is a relative functor by Theorem~\ref{intc=simpc}, and because Reedy fibrant replacement 
preserves and reflects homotopy-equivalences. The lower vertical arrow $\mathbb{R}$ is a relative functor that reflects
DK-equivalences, because Reedy fibrant replacements are levelwise homotopy-equivalences and both fully faithfulness as well as 
essential surjectivity are preserved and reflected along levelwise pairs of levelwise homotopy-equivalences between Segal objects by 
Lemma~\ref{lemmaDKequivinv}. 
The square is a pullback square again by Theorem~\ref{intc=simpc}.

Part 2 now follows from Corollary~\ref{thm5sum1}, the fact that the composite square (\ref{mainsquare2}) is a pullback square,
and that $\mathbb{R}\colon\mathrm{S}(\mathbb{C})^{\times}\rightarrow\mathrm{S}(\mathbb{C})_f^{\times}$ reflects DK-equivalences.
\end{proof}

\begin{corollary}
The univalent completion $u(p)$ of a Kan fibration $p\in\mathbf{S}$ induces an acyclic cofibration
$X_p\rightarrow X_{u(p)}$ in the model structure $(s\mathbf{S},\mathrm{CS})$ of complete 
Segal spaces. In other words, the complete Segal space $X_{u(p)}$ is a fibrant replacement of the Segal space
$X_p$ in Rezk's model structure $(s\mathbf{S},\mathrm{CS})$ for complete Segal spaces.
\end{corollary}

\begin{proof}
By \cite[Theorem 7.7]{rezkhtytps}, the weak equivalences between Segal spaces in the model structure
$(s\mathbf{S},\mathrm{CS})$ for complete Segal spaces are exactly the Dwyer-Kan equivalences.
\end{proof}

\begin{remark}\label{remend}
We have seen that if one starts with a univalent ambient fibration $\pi\colon\tilde{U}\rightarrow U$, then univalent 
completion of $\pi$-small fibrations exists
and induces a Rezk-completion of associated Segal objects. Without the assumption of univalence of $\pi$, criteria for 
the existence of small univalent completions are hard to find. Contrarily, Rezk-completion is often constructed as a 
fibrant replacement functor as referred to in Remark~\ref{remrezkcomplexples}.

Hence, the natural question arises whether univalent completion can be understood as the fibrant replacement in some 
homotopical structure on the category of fibrations in $\mathbb{C}$.
Indeed, it may be interesting to note that both Rezk and Barwick first prove the existence of Rezk-completion in order 
to show that complete Segal objects are indeed the DK-local objects.
\end{remark}

\begin{remark}\label{remintrezkcompl}
In \cite[Theorem 8.5]{aksrezkcompl}, the authors show that every precategory in the syntax of Homotopy Type Theory has 
a Rezk-completion. The tools used in their syntactic proof are basically the same ones used in the proof of 
Proposition~\ref{propexunivcomp}. In particular, it may be interesting to note that \cite[Theorem 8.5]{aksrezkcompl} 
does not generate Rezk-completions ex nihilo either, but assumes the existence of an ambient univalent type universe as well. However 
note that Segal objects in a type theoretic fibration category $\mathbb{C}$ are not the same as precategories in the internal type 
theory $\mathcal{T}_{\mathbb{C}}$.
Yet, whenever a type family $p\colon E\twoheadrightarrow B$ is $0$-truncated (i.e.\ its path-object fibration 
$P_BE\twoheadrightarrow E\times_B E$ is $(-1)$-truncated) and $\mathbb{C}$ satisfies function 
extensionality (see Section~\ref{secsubBMloc}), the type family $\mathrm{Fun}(p)\twoheadrightarrow B\times B$ yields a 
precategory $\mathfrak{C}(p)$ in $\mathcal{T}_{\mathbb{C}}$ via Proposition~\ref{propfunicat} as well. Whenever the univalent 
completion $u(p)$ from Proposition~\ref{propexunivcomp} is again $0$-truncated (that holds e.g.\ whenever the ambient univalent 
fibration $\pi$ itself is $0$-truncated)
it follows along the lines of Proposition~\ref{propcompletion} that the Rezk-completion of this precategory $\mathfrak{C}(p)$ 
in the sense of \cite{aksrezkcompl} is given by the precategory $\mathfrak{C}(u(p))$ given by
$\mathrm{Fun}(u(p))\twoheadrightarrow u(B)\times u(B)$. 
\end{remark}

\subsection{Univalence as a locality condition}\label{secsubBMloc}

Motivated by Rezk's and Barwick's characterization of complete Segal objects as the DK-local Reedy fibrant Segal objects 
in a large class of model categories, we close this paper with a characterization of the $\pi$-small univalent 
fibrations as the $\pi$-small BM-local fibrations in a large class of type theoretic fibration categories $\mathbb{C}$ whenever $\pi$ 
is univalent itself.

Therefore, in the following we consider fibrations in $\mathbb{C}$ that exhibit the \emph{homotopy lifting property},
and suppose that $\mathbb{C}$ has $(-1)$-truncations and satisfies \emph{function extensionality}. Here, we say that a fibration 
$p\colon E\twoheadrightarrow B$ has the homotopy lifting property if 
for every homotopy $H\colon C\rightarrow PB$ and every map $e\colon C\rightarrow E$ such that $p\circ e=s\circ H$, there is a 
homotopy $K\colon C\rightarrow PE$ such that $s\circ K=e$ and $p\circ t\circ K=t\circ H$ (this is weaker than what is usually 
called the homotopy lifting property in model categories). It will be used to replace homotopy-commutative 
squares between fibrations in $\mathbb{C}$ by strictly commutative squares between the same fibrations up to homotopy.
Function extensionality holds in $\mathbb{C}$ if and only if dependent products along fibrations in $\mathbb{C}$ preserve acyclicity 
of fibrations (\cite[Lemma 5.9]{shulmaninv}). Both additional properties are satisfied whenever $\mathbb{C}$ arises from a logical 
model category $\mathbb{M}$ whose fibrant objects are cofibrant and whose cofibrations are pullback-stable along fibrations. Function 
Extensionality is shown in (\cite[Remark 5.10]{shulmaninv}), and the homotopy lifting property holds in virtue of the equivalence of 
left and right homotopies between bifibrant objects in a model category (\cite[Corollary 1.2.6]{hovey}).

For fibrations $p\colon E\twoheadrightarrow B$, $q\colon X\twoheadrightarrow Y$ in
$\mathbb{C}$, consider the object
\begin{align}\label{equdefP}
P_{q,p}:=\sum_{!_{[Y,B]_{\mathbb{C}}}}\prod_{\pi_1}\mathrm{Equiv}_{B^Y\times Y}(\pi_2^{\ast}X,\mathrm{ev}^{\ast}p)
\end{align}
in $\mathbb{C}$, generalizing the corresponding construction in simplicial sets given in \cite[3.5]{klvsimp}). It 
represents homotopy-cartesian squares from $q$ to $p$ in the sense that for all $C\in\mathbb{C}$, there are natural 
isomorphisms
\begin{align}\label{equrepdefP}
\mathbb{C}(C,P_{p,q})\cong\{(f,w)\mid f\colon C\times Y\rightarrow B, w\colon C\times X\xrightarrow{\simeq} f^{\ast}p\text{ over }C\times Y\}.
\end{align}
A dashed lift of the form
\begin{align}\label{equrepdefP2}
\begin{gathered}
\xymatrix{
 & P_{q,p}\ar@{->>}[d] \\
C\ar[r]_(.4){\ulcorner f\urcorner}\ar@{-->}[ur] & [Y,B]_{\mathbb{C}}
}
\end{gathered}
\end{align}
corresponds exactly to a tuple $(f,w)$ under (\ref{equrepdefP}) such that $\ulcorner f\urcorner$ represents the morphism
$f\colon C\times Y\rightarrow B$.
 
Furthermore, every homotopy-cartesian square $\alpha\colon q_1\Rightarrow q_2$ between fibrations induces a morphism
\[P_{\alpha}\colon P_{q_2,p}\rightarrow P_{q_1,p}\]
in $\mathbb{C}$, given representably by horizontal composition to the right of
$C\times\alpha\colon C\times q_1\Rightarrow C\times q_2$ with elements in $\mathbb{C}(C,P_{q_2,p})$. Here we 
implicitly choose some cleavage for the target fibration $t\colon\mathcal{F}/p\twoheadrightarrow\mathbb{C}/B$ 
(that is, some choice of representatives for the isomorphism classes of pullbacks of $p$ along every map into $B$) to 
make $P_{\alpha}$ a strict natural transformation on representables. It is straight-forward as well to 
construct the objects $P_{q_i,p}$ together with the map $P_{\alpha}$ in the internal type theory of $\mathbb{C}$.

\begin{proposition}\label{propintBMloc}
Suppose $\mathbb{C}$ has $(-1)$-truncations. Let $\pi\colon\tilde{U}\twoheadrightarrow U$ be a univalent fibration in $\mathbb{C}$ and
$p\colon E\twoheadrightarrow B$ be $\pi$-small. 
If for all BM-equivalences $\alpha\colon q_1\Rightarrow q_2$ between $\pi$-small fibrations in $\mathbb{C}$ the induced morphism 
$P_{\alpha}\colon P_{q_2,p}\rightarrow P_{q_1,p}$ is a homotopy-equivalence, then $p$ is univalent. Whenever 
furthermore the $\pi$-small univalent fibrations in $\mathbb{C}$ have the homotopy lifting property and $\mathbb{C}$ satisfies 
function extensionality, then the converse holds as well.  
\end{proposition}
\begin{proof}
Suppose $P_{\alpha}\colon P_{q_2,p}\rightarrow P_{q_1,p}$ is a homotopy-equivalence for all BM-equivalences
$\alpha\colon q_1\Rightarrow q_2$ between small fibrations in $\mathbb{C}$. Let $\gamma\colon p\Rightarrow u(p)$ be a univalent 
completion of $p$. Then, by assumption, the map
\[P_{\gamma}\colon P_{u(p),p}\rightarrow P_{p,p}\]
is a homotopy-equivalence. Let $P_{\gamma}^{-1}\colon P_{p,p}\rightarrow P_{u(p),p}$ be a homotopy-inverse. The identity 
square $1_p$ from $p$ to $p$ is homotopy-cartesian, and hence yields a global section $1_p\in \mathbb{C}(1,P_{p,p})$. We 
obtain $P_{\gamma}^{-1}\circ 1_p\in\mathbb{\mathbb{C}}(1,P_{u(p),p})$, and thus an associated homotopy-cartesian square
$\beta\colon u(p)\Rightarrow p$. Both horizontal compositions of homotopy-cartesian squares
$\gamma\circ\beta\colon u(p)\Rightarrow u(p)$ and $\beta\circ\gamma\colon p\Rightarrow p$ are homotopy-cartesian again.
\begin{align}
\xymatrix{
E\ar[r]\ar@{->>}[d]^{p} & u(E)\ar@{->>}[d]^{u(p)}\ar[r] & E\ar@{->>}[d]^{p}\\
B\ar[r]\ar[r]_{b(\gamma)} & u(B)\ar[r]\ar[r]_{b(\beta)} & B
}
& \hspace{.5cm} &
\xymatrix{
u(E)\ar[r]\ar@{->>}[d]^{u(p)} & E\ar[r]\ar@{->>}[d]^{p} & u(E)\ar@{->>}[d]^{u(p)} \\
u(B)\ar[r]_{b(\beta)} & B\ar[r]_{b(\gamma)} & u(B) 
}
\end{align}
By univalence of $u(p)$, the bottom composition $u(B)\rightarrow B\rightarrow u(B)$ is homotopic to the identity on $B$ 
(as both maps represent the same pullback $u(p)$). If we can show that the composition $B\rightarrow u(B)\rightarrow B$ 
is homotopic to the identity as well, then the map $B\rightarrow u(B)$ is a homotopy-equivalence, and hence so is its
homotopy-pullback $E\rightarrow u(E)$. It follows that the fibrations $p$ and $u(p)$ are homotopy-equivalent. Since 
$u(p)$ is univalent, and univalence is invariant under homotopy-equivalence (\cite[Corollary 1.6.4]{thesis}), so is $p$. 

We are thus left to show that the composition $B\rightarrow u(B)\rightarrow B$ is homotopic to the identity. By 
construction of the object $P_{p,p}$ in (\ref{equdefP}) and (\ref{equrepdefP2}), the canonical fibrations
$t\colon P_{p,p}\twoheadrightarrow [B,B]_{\mathbb{C}}$ and $P_{u(p),p}\twoheadrightarrow [u(B),B]_{\mathbb{C}}$ 
are such that the diagram
\[\xymatrix{
\ast\ar[r]^(.4){P_{\gamma}^{-1}\circ 1_p}\ar@/_/[dr]_(.4){\ulcorner b(\beta)\urcorner} & P_{u(p),p}\ar@{->>}[d]\ar[r]^{P_{\gamma}} & P_{p,p}\ar@{->>}[d] \\
 & [u(B),B]_{\mathbb{C}}\ar[r]_(.55){b(\gamma)^{\ast}} & [B,B]_{\mathbb{C}}
}\]
commutes. By assumption, we have a homotopy
$H\colon P_{\gamma}P_{\gamma}^{-1}\sim 1_{P_{p,p}}$ and hence a diagram as follows.
\[\xymatrix{
 & & P(P_{p,p})\ar@{->>}[d]\ar[r]^{\text{tr}_t} & P([B,B]_{\mathbb{C}})\ar@{->>}[d] \\
\ast\ar[rr]_(.35){(P_{\gamma}P_{\gamma}^{-1}1_p,1_p)}\ar@/^/[urr]^{H\circ 1_p} & & P_{p,p}\ar@{->>}[r]_(.4){t\times t}\times P_{p,p} & [B,B]_{\mathbb{C}}\times [B,B]_{\mathbb{C}}
}\]
The composition $\mathrm{tr}_t\circ H\circ 1_p$ is a homotopy between the global sections $t\circ 1_p=1_B$ and
$t\circ P_{\gamma}P_{\gamma}^{-1}\circ 1_p = b(\beta)\circ b(\gamma)$ of $[B,B]_{\mathbb{C}}$. This yields a homotopy between the maps 
$1_B\colon B\rightarrow B$ and $b(\beta)\circ b(\gamma)\colon B\rightarrow B$ as claimed, using that there is always a map 
$P([B,B]_{\mathbb{C}})\rightarrow[B,P(B)]_{\mathbb{C}}$ over $[B,B]_{\mathbb{C}}\times [B,B]_{\mathbb{C}}$ (given by the term 
``happly" in \cite[Section 5]{shulmaninv}). This proves the first part of the proposition.

For the other direction, we use that under the assumption of function extensionality a fibration $p$ in $\mathbb{C}$ is univalent if 
and only if for every fibration $q$ in $\mathbb{C}$, the object $P_{q,p}$ is $(-1)$-truncated in $\mathbb{C}$
(\cite[Theorem 1.5.2]{thesis}, generalizing \cite[Theorem 3.5.3]{klvsimp}). 
Thus, suppose $p$ is univalent and let $\alpha\colon q_1\Rightarrow q_2$ be a BM-equivalence between $\pi$-small 
fibrations in $\mathbb{C}$. We want to show that $P_{\alpha}\colon P_{q_2,p}\rightarrow P_{q_1,p}$ is a homotopy-equivalence. Since 
$p$ is univalent, the objects $P_{q_i,p}$ are $(-1)$-truncated, and so any morphism $P_{q_1,p}\rightarrow P_{q_2,p}$ is a homotopy-
inverse to $P_{\alpha}$ (e.g.\ by Lemma~\ref{lemmamono}). It thus suffices to construct some map of this type. We will do so 
representably, and we will do so in the homotopy category $\text{Ho}(\mathbb{C})$ associated to $\mathbb{C}$, whose objects are 
exactly the objects of $\mathbb{C}$, and whose hom-sets are the quotients
\[\text{Ho}(\mathbb{C})(C,D)=\mathbb{C}(C,D)/\sim\]
by the homotopy relation in $\mathbb{C}$ (\cite[Section 3.3]{joyaltribes}). 
Therefore, for all $C\in\mathbb{C}$, we are to construct a map
\[k_C\colon\mathbb{C}(C,P_{q_1,p})_{/\sim}\rightarrow\mathbb{C}(C,P_{q_2,p})_{/\sim}\]
such that for all $f\colon C\rightarrow D$ in $\mathbb{C}$, the associated square commutes. Since both $P_{q_i,p}$ are
$(-1)$-truncated, the hom-sets $\mathbb{C}(C,P_{q_i,p})_{/\sim}$ are either empty or a singleton. Hence, naturality of the maps $k_C$ 
is a triviality.

Given a map $f\colon C\rightarrow P_{q_1,p}$, we are thus only to construct an element in $\mathbb{C}(C,P_{q_2,p})$. Therefore, 
consider the resulting diagram of solid arrows
\begin{align}\label{diagpropintBMloc}
\begin{gathered}
\xymatrix{
C\times X_1\ar@{->>}[dd]_{C\times q_1}\ar[dr]\ar[rr] & & E\ar@{->>}[dd]|\hole^(.7){p}\ar[dr] & \\ 
 &  C\times X_2\ar@{->>}[dd]_(.3){C\times q_2}\ar[rr]\ar@{-->}[ur]\ar[rr] & & \tilde{U}\ar@{->>}[dd]^{\pi} \\
C\times Y_1\ar[dr]\ar[rr]|(.56)\hole & & B\ar[dr]	& \\
 & C\times Y_2\ar[rr]\ar@{-->}[ur] & & U \\
}
\end{gathered}
\end{align}
in $\mathbb{C}$, where all four sides between the vertical fibrations are homotopy-cartesian. Here, the back square corresponds to the 
map $f$ under (\ref{equrepdefP}), the left hand side square is $C\times\alpha$, the front and the right hand side squares 
are given by smallness of the fibrations $p$ and $q$. The bottom square (and hence the top square) commutes up to 
homotopy in $\mathbb{C}$ by univalence of $\pi$ (by the same argument we applied to (\ref{diagunivcompletionunique})). 
Since $(-1)$-connected maps are pullback-stable along fibrations and stable under composition with homotopy-equivalences, both maps 
$C\times Y_1\rightarrow C\times Y_2$ and $C\times X_1\rightarrow C\times X_2$ are
$(-1)$-connected. The map $B\rightarrow U$ is $(-1)$-truncated by 
Proposition~\ref{propmereproppb}, and again so is the top map $E\rightarrow\tilde{U}$. We thus obtain the two dashed 
arrows in (\ref{diagpropintBMloc}) by Lemma~\ref{lemmahmtyimageinvariance}.2. The resulting square from $C\times q_2$ to 
$p$ commutes up to homotopy via Lemma~\ref{lemmamono}, as it does so after postcomposition with the $(-1)$-truncated map 
$B\rightarrow U$. By the homotopy lifting property in $\mathbb{C}$, we may replace its top map $C\times X_2\rightarrow E$ up to 
homotopy by another map $C\times X_2\rightarrow E$ such that the resulting square over $C\times Y_2\rightarrow B$ commutes 
\emph{strictly}. This strictly commutative square is homotopy-cartesian, because the front square and the right hand side square in 
(\ref{diagpropintBMloc}) are. It hence represents an element $k_C(f)\in\text{Ho}(\mathbb{C})(C,P_{q_2,p})$.
\end{proof}

We end this paper with three auxiliary remarks on Proposition~\ref{propintBMloc}.

\begin{remark}
Regarding the two additional assumptions on $\mathbb{C}$ in the second part of 
Proposition~\ref{propintBMloc}, the author is not aware of any conceptual reasons why either would be strictly necessary.
However, any proof of homotopy-invertibility of the maps  $P_{\alpha}\colon P_{q_2,p}\rightarrow P_{q_1,p}$ requires computation of 
the path-objects of the $P_{q_i,p}$, and therefore computation of identities of the function types $Y_i\rightarrow B$ occurring 
therein. Function extensionality determines precisely such identities. Since univalence of a universe $V$ that is closed under 
dependent sum, dependent product and identity type formation implies function extensionality for all type families captured by it 
(\cite[Section 4.9]{hott}), function extensionality as needed in Proposition~\ref{propintBMloc} is hence satisfied whenever the bases 
$Y_i\twoheadrightarrow 1$ and $U\twoheadrightarrow 1$ themselves are small with respect to some univalent fibration that is closed 
under all respective type formers.
\end{remark}

\begin{remark}
The formulation of Proposition~\ref{propintBMloc} by internal means in $\mathbb{C}$ allows us to omit a lengthy discussion of
associated higher categorical structures in the underlying $(\infty,1)$-category of $\mathbb{C}$. 
This remark is a sketch of the resulting higher categorical implications to be expected. Therefore, recall that every fibration 
category $\mathbb{C}$ comes equipped with an underlying quasi-category $\mathcal{C}:=\text{Ho}_{\infty}(\mathbb{C})$ which can be 
constructed by way of Szumi\l{}o's framed nerve construction (\cite[3]{szumilococomplqcats}), or by way of the more classical 
simplicial localization functors applied to the relative category $(\mathbb{C},\mathcal{W})$ (\cite{kapszuframes}). If $\mathbb{C}$ 
comes from a logical model category with cofibrant fibrant objects, its hom-spaces are equivalent to the classical 
function complexes obtained by frames or coframes in $\mathbb{M}$ (\cite{dkfcs}). It follows from 
the work of \cite{szumilococomplqcats} and \cite[4,5]{kapulkinlccc} that the $(\infty,1)$-localization functor
$\mathbb{C}\rightarrow\text{Ho}_{\infty}(\mathbb{C})$ of a type theoretic fibration category preserves pullbacks, 
dependent products and sums along fibrations. Thus, by the formula in (\ref{equdefP}), the object
$P_{q,p}\in\mathcal{C}$ presents the space-valued presheaf $\mathcal{C}^{op}\rightarrow\mathcal{S}$ of cartesian 
squares from $q\times\blank$ to $p$ in $\mathcal{C}$. Using that all objects in a type theoretic fibration category are 
cofibrant, and that $(-1)$-connected maps are homotopy-pullback-stable in $\mathcal{C}$, one can then use
Proposition~\ref{propintBMloc} to show that the small univalent maps in $\mathbb{C}$ are exactly the small maps that 
are local for the small BM-equivalences in the underlying quasi-category of $\mathbb{C}$.
It follows that univalent completion as defined in Proposition~\ref{propexunivcomp} yields a left adjoint to the fully 
faithful inclusion
\[\mathcal{U}^{\times}_{\pi}\hookrightarrow(\mathcal{C}^{\Delta^1})^{\times}_{\pi},\]
where $(\mathcal{C}^{\Delta^1})^{\times}_{\pi}\subset\mathcal{C}^{\Delta^1}$ is the subcategory generated by the small arrows in
$\mathcal{C}$ and cartesian squares between them, and $\mathcal{U}^{\times}_{\pi}$ is spanned by the small univalent maps in
$\mathcal{C}$.
\end{remark}

\begin{remark}
One can show that for every triple of fibrations $p,q,r$ in $\mathbb{C}$ there is an arrow
$\circ\colon P_{r,q}\times P_{q,p}\rightarrow P_{r,p}$ given by concatenation of homotopy-cartesian squares and a unit
$1_p\colon\ast\rightarrow P_{p,p}$ given by the identity square. The associativity and unitality diagrams can be shown 
to commute up to homotopy. It follows that the assignment $P$ yields a
$(\text{Ho}(\mathbb{C}),\times)$-enriched category $(\mathcal{F},P,1,\circ)$. Furthermore, every pointwise homotopy-equivalence 
$w\colon r\rightarrow q$ induces homotopy-equivalences $w^{\ast}\colon P_{q,p}\rightarrow P_{r,p}$ and
$w_{\ast}\colon P_{p,r}\rightarrow P_{p,q}$.
One may thus ask whether this assignment more generally induces an enrichment of $(\mathcal{C}^{\Delta^1})^{\times}$ in 
the cartesian closed $(\infty,1)$-category $\mathcal{C}$, such that
$\mathcal{C}(1,P_{q,p})\simeq(\mathcal{C}^{\Delta^1})^{\times}(q,p)$ for all $q,p\in \mathcal{C}^{\Delta^1}$. 
\end{remark}

\bibliographystyle{amsplain}
\bibliographystyle{alpha}
\bibliography{BSBib}
\Address

\end{document}